\documentclass[a4paper,11pt,reqno]{amsart}
\usepackage{amsmath}
\usepackage{amsfonts}
\usepackage{amsthm}
\usepackage{amssymb}

\usepackage{graphicx}
\usepackage{color}
\usepackage{subfigure}

\usepackage{setspace}
\usepackage[active]{srcltx}
\usepackage{mdwlist}

\usepackage{hyperref}
\usepackage{enumerate}
\usepackage{mathrsfs}
\usepackage{mdwlist}

\usepackage{textcomp}
\usepackage{wasysym}

\newcommand{\R}{\mathbb{R}}

\theoremstyle{plain}
\newtheorem{theorem}{Theorem}
\newtheorem{lemma}[theorem]{Lemma}
\newtheorem{definition}[theorem]{Definition}

\newtheorem{corollary}[theorem]{Corollary}
\newtheorem{prop}[theorem]{Proposition}
\theoremstyle{remark}
\newtheorem{remark}[theorem]{Remark}

\newtheorem{example}[theorem]{Example}

 \author[A. Ciomaga]{Adina Ciomaga$^\ddag$}
\address{$^\ddag$Centre de Math\'ematiques et de Leurs
  Applications, Ecole Normale Sup\'erieure de Cachan, CNRS, UniverSud,
  61 avenue du pr\'esident Wilson, F-94230 Cachan, France}
\email{ciomaga@cmla.ens-cachan.fr}

\title[On the Strong Maximum Principle for PIDEs]{On the Strong Maximum Principle for Second Order Nonlinear Parabolic Integro-Differential Equations}

\begin{document}
\graphicspath{{./figures/}}

\begin{abstract}
This paper is concerned with the study of the Strong Maximum Principle for semicontinuous viscosity solutions of fully nonlinear, second-order parabolic integro-differential equations. 
We study separately the propagation of maxima in the horizontal component of the domain and the local vertical propagation in simply connected sets of the domain. 
We give two types of results for horizontal propagation of maxima: one is the natural extension of the classical results of local propagation of maxima and the other  comes from the structure of the nonlocal operator.
As an application, we use the Strong Maximum Principle to prove a Strong Comparison Result of viscosity sub and supersolution for integro-differential equations.
\end{abstract} \medskip

\maketitle

{\textbf{Keywords:} nonlinear parabolic integro-differential equations, strong maximum principle, viscosity solutions \medskip}
 
{\textbf{AMS Subject Classification:} 35R09,  35K55 , 35B50, 35D40 \medskip}

\bigskip\section{Introduction}\bigskip

We investigate the Strong Maximum Principle for viscosity solutions of second-order non-linear parabolic integro-differential equations of the form

\begin{equation}\label{ePIDEs}
 u_t + F(x,t,Du,D^2u,\mathcal{I}[x,t,u]) =0 \hbox{ in } \Omega\times (0,T)
\end{equation}
where $\Omega\subset\R^N$ is an open bounded set, $T>0$ and $u$ is a real-valued function defined on $\R^N\times[0,T]$.
The symbols $u_t$, $Du$, $D^2u$ stand for the derivative with respect to time, respectively the gradient 
and the Hessian matrix with respect to $x$. $\mathcal{I}[x,t,u]$ is an integro-differential operator, 
taken on the whole space $\R^N$.  Although the nonlocal operator is defined on the whole space, 
we consider equations on a bounded domain $\Omega$. Therefore, we assume that the function 
$u = u(x,t)$ is a priori defined outside the domain $\Omega$. The choice corresponds to prescribing 
the solution in $\Omega^c\times(0,T)$, as for example in the case of Dirichlet boundary conditions.\smallskip

The nonlinearity $F$ is a real-valued, continuous function in $\Omega\times [0,T]\times\R^N\times\mathbb{S}^N\times\R$, ($\mathbb{S}^N$ being the set of real symmetric $N\times N$ matrices) and \emph{degenerate elliptic}, i.e.
\begin{equation}\label{Ell}
F(x,t,p,X,l_1)\leq  F(x,t,p,Y,l_2) \hbox{ if } X\geq Y, \ l_1\geq l_2,
\end{equation}
for all $(x,t)\in\overline{\Omega}\times [0,T]$, $p\in\R^N\setminus\{0\}$, $X,Y\in\mathbb{S}^N$ and $l_1,l_2\in\R$.\smallskip

Throughout this work, we consider integro-differential operators of the type
\begin{equation}\label{NL_op}
 \mathcal{I}[x,t,u]=\int_{\R^N} \left(u(x+z,t)-u(x,t)-Du(x,t)\cdot z 1_B(z)\right)\mu_x(dz)
\end{equation}
where $1_B(z)$ denotes the indicator function of the unit ball $B$ and $\{\mu_x\}_{x\in\Omega}$ is a family of L\'evy measures, 
i.e. non-negative, possibly singular, Borel measures on $\Omega$ such that
$$\sup_{x\in\Omega}\int_{\R^N}\min(|z|^2,1)\mu_x(dz)<\infty.$$
In particular, L\'evy-It\^o operators are important special cases of nonlocal operators and are defined as follows
\begin{equation}\label{LI_op}
 \mathcal{J}[x,t,u]=\int_{\R^N}\left(u(x+j(x,z),t)-u(x,t)-Du(x,t)\cdot j(x,z) 1_B(z)\right)\mu(dz)
\end{equation}
where $\mu$ is a L\'evy measure and  $j(x,z)$ is the size of the jumps at $x$ satisfying
$$|j(x,z)|\leq C_0 |z|, \ \forall x\in\Omega, \forall z\in\R^N$$
with $C_0$ a positive constant. \smallskip

We denote by $USC(\R^N\times[0,T])$ and $LSC(\R^N\times[0,T])$ the set of respectively upper and lower semi-continuous functions in $\R^N\times[0,T]$. 
By Strong Maximum for equation (\ref{ePIDEs}) in an open set $\Omega\times(0,T)$ we mean the following.   \smallskip

\medskip\noindent
SMaxP: \emph{any $u\in USC(\R^N\times[0,T])$ viscosity subsolution of (\ref{ePIDEs}) that attains a maximum at $(x_0,t_0)\in\Omega\times(0,T)$ is constant in $\Omega\times[0,t_0]$}. \medskip

The Strong Maximum Principle follows from the horizontal and vertical propagation of maxima, that we study separately.
By horizontal propagation of maxima we mean the following:
if the maximum is attained at some point $(x_0,t_0)$ then the function becomes constant in the connected component of the domain $\Omega\times\{t_0\}$ which contains the point $(x_0,t_0)$. By local vertical propagation we understand that if the maximum is attained at some point $(x_0,t_0)$ then at any time $t<t_0$ one can find another point $(x,t)$ where the maximum is attained. This will further imply the propagation of maxima in the region $\Omega\times (0, t_0)$. \smallskip

\begin{figure}[h!]
\centering
\includegraphics[width=0.95\linewidth]{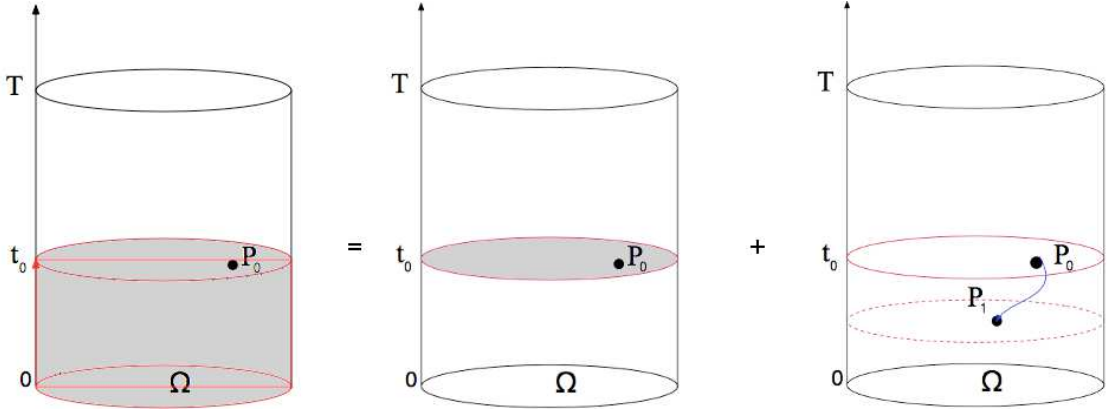} 
\caption{\small Strong Maximum Principle follows from the horizontal and vertical propagation of maxima.}
\label{fig::smp}
\end{figure}

We set $Q_T = \Omega \times (0, T ]$ and for any point $P_0 = (x_0 , t_0 ) \in Q_T$, we denote by $S(P_0)$ the set of all points $Q \in Q_T$ which can be connected to $P_0$ by a simple continuous curve in $Q_T$ and by $C(P_0)$ we denote the connected component of $\Omega\times\{ t_0 \}$ which contains $P_0$. \smallskip

The horizontal propagation of maxima in $C(P_0)$ requires two different perspectives. An almost immediate result follows from the structure of the nonlocal operator. More precisely, we show that \textit{Strong Maximum Principle holds for PIDEs involving nonlocal operators in the form (\ref{NL_op}) whenever the whole domain (not necessarily connected) can be covered by translations of measure supports, starting from a maximum point.} This is the case for example of a pure nonlocal diffusion
\begin{equation*}
u_t - \mathcal{I}[x,t,u]=0 \hbox{ in } \R^N \times (0,T)
\end{equation*}
where $\mathcal{I}$ is an isotropic L\'evy operator of form (\ref{NL_op}), integrated against the L\'evy measure associated with the fractional Laplacian $(-\Delta)^{\beta/2}, \beta\in(0,2)$:
$$ \mu(dz) = \frac{dz}{|z|^{N+\beta}}. $$
The result is the natural extension to PIDEs of the maximum principle for nonlocal 
operators generated by nonnegative kernels obtained by Coville in \cite{C:08:SMaxNL}.\smallskip

Nevertheless, there are equations for which maxima do not propagate just by translating measure supports, such as pure nonlocal equations with nonlocal terms associated with the fractional Laplacian, but whose measure supports are defined only on half space. \emph{Mixed integro-differential equations}, i.e. equations for which local diffusions occur only in certain directions and nonlocal diffusions on the orthogonal ones cannot be handled by simple techniques, as they might be degenerate in both local or nonlocal terms but the overall behavior might be driven by their interaction (the two diffusions cannot cancel simultaneously).  We have in mind equations of the type
\begin{equation*}
 u_t  - \mathcal{I}_{x_1}[u] - \frac{\partial^2u}{\partial x_2^2}= 0 \hbox{ in } \R^2\times (0,T)
\end{equation*}
for $x=(x_1,x_2)\in\R^2$. 
The diffusion term gives the ellipticity in the direction of $x_2$, while the nonlocal term gives it in the direction of $x_1$ 
$$
\mathcal{I}_{x_1}[u] = \int_{\R} \left(u(x_1+z_1,x_2)-u(x)-\frac{\partial u}{\partial x_1}(x)\cdot z_1 1_{[-1,1]}(z_1)\right)\mu_{x_1}(dz_1)
$$
where $\{\mu_{x_1}\}_{x_1}$ is a family of L\'evy measures.
However, we manage to show that under some nondegeneracy and scaling assumptions on the nonlinearity $F$, 
\textit{if a viscosity subsolution attains a maximum at $P_0=(x_0,t_0)\in Q_T$, 
then $u$ is constant (equal to the maximum value) in the horizontal component $C(P_0)$.}\smallskip

We then prove the local propagation of maxima in the cylindrical region 
$\Omega\times(0,T]$ and thus extend to parabolic integro-differential equations 
the results obtained by Da Lio in \cite{daLio:04:SMaxP} and Bardi and Da Lio 
in \cite{BarLio:01:SMaxP1} and \cite{BarLio:03:SMaxP2} for fully nonlinear 
degenerate elliptic convex and concave Hamilton Jacobi operators. For helpful details of 
Strong Maximum Principle results for Hamilton Jacobi equations we refer to 
\cite{Barles:94:ViscositeHJ}. Yet, it is worth mentioning that Strong Maximum 
Principle for linear elliptic equations goes back to Hopf in the 20s and to 
Nirenberg, for parabolic equations \cite{Nirenberg:53:SMaxP}. 
\smallskip

In the last part we use Strong Maximum Principle to prove a Strong Comparison Result of viscosity sub and supersolution for integro-differential equations of the form (\ref{ePIDEs}) with the Dirichlet boundary condition
\begin{equation*}
 u = \varphi \hbox{ on } \Omega^c\times [0,T]
\end{equation*}
where $\varphi$ is a continuous function. \smallskip

Nonlocal equations find many applications in mathematical finance and occur in the theory of L\'evy jump-diffusion processes. 
The theory of viscosity solutions has been extended for a rather long time to Partial Integro-Differential Equations (PIDEs). 
Some of the first papers are due to Soner \cite{Soner:86:OC1}, \cite{Soner:86:OC2}, 
in the context of stochastic control jump diffusion processes. 
Following his work, existence and comparison results of solutions for 
\emph{first order PIDEs} were given by Sayah in \cite{Sayah:91:HJID1} and \cite{Sayah:91:HJID2}.\smallskip

\emph{Second-order degenerate PIDEs} are more complex and required careful studies, 
according to the nature of the integral operator (often reflected in the singularity 
of the L\'evy measure against which they are integrated). When these equations involve 
\emph{bounded integral operators}, general existence and comparison results for 
semi-continuous and unbounded viscosity solutions were found by Alvarez and Tourin 
\cite{AlvTou:96:ViscPIDE}. Amadori extended the existence and uniqueness results to a 
class of Cauchy problems for integro-differential equations, starting with initial data 
with exponential growth at infinity \cite{Amadori:03:viscPIDEs} and proved a local Lipschitz regularity result.\smallskip

Systems of parabolic integro-differential equations dealing with \emph{second order nonlocal operators} 
were connected to backwards stochastic differential equations in \cite{BarlesBuckPard:97:SDE_PIDEs} 
and existence and comparison results were established. Pham connected the optimal stopping time 
problem in a finite horizon of a controlled jump diffusion process with a parabolic PIDE in 
\cite{Pham:98:OCJDPviscosity} and proved existence and comparison principles of uniformly 
continuous solutions. Existence and comparison results were also provided by Benth, Karlsen and 
Reikvam in \cite{BeKaRe:01:OptimalSelectionPIDEs} where a singular stochastic control problem 
is associated to a nonlinear second-order degenerate elliptic integro-differential equation 
subject to gradient and state constraints, as its corresponding Hamilton-Jacobi-Bellman equation.\smallskip

Jakobsen and Karlsen in \cite{JakoKarl:06:MPIDE} used the original approach due to Jensen \cite{Jensen:88:MaxP_PDEs}, Ishii \cite{Ishii:89:UE_viscosityPDEs}, Ishii and Lions \cite{IL:90:ViscPDE}, Crandall and Ishii \cite{CrandallIshii:90:MaxP} and Crandall, Ishii and Lions \cite{CIL:92:UsersGuide} for proving comparison results for 
viscosity solutions of nonlinear degenerate elliptic integro-partial differential equations with second 
order nonlocal operators. Parabolic versions of their main results were given in \cite{JakoKarl:05:ContDep_PIDEs}. 
They give an analogous of Jensen-Ishii's Lemma, a keystone for many comparison principles, but they are restricted 
to \emph{subquadratic solutions}. \smallskip

The viscosity theory for general PIDEs has been recently revisited and extended to 
\emph{solutions with arbitrary growth at infinity} by Barles and Imbert \cite{BI:08:ViscNL}. 
The authors provided as well a variant of Jensen Ishii's Lemma for general integro-differential equations. 
The notion of viscosity solution generalizes the one introduced by Imbert in \cite{Imbert:05:NLreg} 
for first-order Hamilton Jacobi equations in the whole space and Arisawa in \cite{Arisawa:06:ViscosityPIDEs}, 
\cite{Arisawa:06:ViscosityPIDEsCor} for degenerate integro-differential equations on bounded domains.\smallskip

The paper is organized as follows. In section \S\ref{sec:StrongMaxPrinc} 
we study separately the propagation of maxima in $C(P_0)$ and in the region $\Omega\times (0, t_0)$. 
In section \S\ref{sec:StrongMaxPrinc_LI} similar results are given for L\'evy It\^o operators. 
Examples are provided in section \S\ref{sec:examples}. In section \S\ref{sec:StrongCompPrinc} 
we prove a Strong Comparison Result for the Dirichlet Problem, based on the Strong Maximum Principle for the linearized equation.
\smallskip

\bigskip
\section{Strong Maximum Principle - General Nonlocal Operators}\label{sec:StrongMaxPrinc}

The aim of this section is to prove the local propagation of maxima of viscosity solutions of (\ref{ePIDEs}) in the cylindrical region $Q_T$. 
As announced, we study separately the propagation of maxima in the horizontal domains $\Omega\times\{t_0\}$ and the local vertical propagation in regions $\Omega\times(0,t_0)$. Each case requires different sets of assumptions. \smallskip

In the sequel, we refer to integro-differential equations of the form (\ref{ePIDEs})
where the function $u$ is a priori given outside $\Omega$. Assume that $F$ satisfies 

\begin{itemize}
\item [$(E)$] $F$ is continuous in $\Omega\times [0,T]\times\R^N\times\mathbb{S}^N\times\R$ and degenerate elliptic.
\end{itemize} 

\noindent Results are presented for general nonlocal operators 
\begin{eqnarray}\nonumber
 \mathcal{I}[x,t,u]=\int_{\R^N} (u(x+z,t)-u(x,t)-Du(x,t)\cdot z 1_B(z))\mu_x(dz)
\end{eqnarray}
where $\{\mu_x\}_{x\in\Omega}$ is a family of L\'evy measures.
We assume it satisfies assumption  
\begin{itemize}
\item [($M$)]  there exists a constant $\tilde C_\mu>0$ such that, for any $x\in\Omega$, 
$$
\int_B|z|^2\mu_x(dz)+\int_{\R^N\setminus B}\mu_x(dz)\leq \tilde C_\mu.
$$ 
\end{itemize}\smallskip

To overcome the difficulties imposed by the behavior at infinity of the measures $(\mu_x)_x$, we often need to split the nonlocal term into 
\begin{eqnarray}
\nonumber
 \mathcal{I}^1_\delta [x,t,u] & = &\int_{|z|\leq \delta}(u(x+z,t)-u(x,t)-Du(x,t)\cdot z 1_B(z))\mu_x(dz) \\
\nonumber
 \mathcal{I}^2_\delta [x,t,p,u]& = &\int_{|z|>\delta}(u(x+z,t)-u(x,t)-p\cdot z 1_B(z))\mu_x(dz)
\end{eqnarray}

with $0<\delta<1$ and $p\in\R^N$. \smallskip

There are several equivalent definitions of viscosity solutions, but we will mainly refer to the following one.\smallskip

\begin{definition}[Viscosity solutions]
An usc function $u:\R^N\times[0,T]\rightarrow\R$ is a \emph{subsolution} of (\ref{ePIDEs}) 
if for any $\phi\in C^2(\R^N\times[0,T])$ such that $u-\phi$ attains a global maximum at $(x,t)\in\Omega\times (0,T)$  
$$
\phi_t(x,t) +  F(x,t,\phi(x,t),D\phi(x,t),D^2\phi(x,t), \mathcal{I}^1_\delta[x,t,\phi]+\mathcal{I}^2_\delta[x,t,D\phi(x,t),u])\leq 0.
$$
A lsc function $u:\R^N\times[0,T]\rightarrow\R$ is a \emph{supersolution} of (\ref{ePIDEs}) 
if for any test function $\phi\in C^2(\R^N\times[0,T])$ such that $u-\phi$ attains a global minimum at $(x,t)\in\Omega\times(0,T)$ 
$$
\phi_t(x,t) +  F(x,t,\phi(x,t),D\phi(x,t),D^2\phi(x,t), \mathcal{I}^1_\delta[x,t,\phi]+\mathcal{I}^2_\delta[x,t,D\phi(x,t),u])\geq 0.
$$
\end{definition}

\medskip\subsection[Horizontal Propagation - Translations of Measure Supports]
{Horizontal Propagation of Maxima by Translations of Measure Supports}

Maximum principle results for nonlocal operators generated by nonnegative kernels 
defined on topological groups acting continuously on a Hausdorff space were settled out 
by Coville in \cite{C:08:SMaxNL}. In the following, we present similar results for integro-differential 
operators in the setting of viscosity solutions.\smallskip
 
It can be shown that Maximum Principle holds for nonlocal operators given by (\ref{NL_op}) 
whenever the whole domain can be covered by translations of measure supports, starting from a maximum point, as suggested in Figure \ref{fig::translations}. 

An additional assumption is required with respect to the nonlinearity $F$. More precisely we require that 
\begin{itemize}
\item [$(E')$] $F$ is continuous, degenerate elliptic and for $x,p\in\R^N$ and $l\in\R$
\begin{equation*}
F(x,t,0,O,l)\leq 0 \Rightarrow l\geq 0. 
\end{equation*}
\end{itemize} 

For the sake of precision, the following result is given for integro-differential equations defined in $\R^N$. We explain in Remark \ref{rk::Omega} what happens when we restrict to some open set $\Omega$.

\begin{theorem}\label{th::trans_supp}
Assume the family of measures $\{\mu_x\}_{x\in\Omega}$ satisfies assumption $(M)$. Let $F$ satisfy $(E')$ in $\R^N\times [0,T]$ and $u\in USC(\R^N\times[0,T])$ be a viscosity subsolution of (\ref{ePIDEs}) in $\R^N\times (0,T)$. If $u$ attains a global maximum at $(x_0,t_0)\in\R^N\times(0,T)$, then $u(\cdot,t_0)$ is constant on $\overline{\bigcup_{n\geq 0}A_n},$ with
\begin{equation}\label{trans_supp}
A_0 = \{x_0\}, \hspace{0.2cm} A_{n+1} = \bigcup_{x\in A_{n}}(x + \hbox{supp}(\mu_x)).
\end{equation}
\end{theorem}

\begin{proof}
Assume that $u$ is a viscosity subsolution for the given equation. 
Consider the test-function $\psi \equiv 0$ and write the viscosity inequality at point $(x_0,t_0)$
$$
F(x_0,t_0,0,O,\mathcal{I}^1_\delta[x_0,t_0,\psi]+\mathcal{I}^2_\delta[x_0,t_0,D\psi(x_0,t_0),u])\leq 0.
$$
This implies according to assumption $(E')$, that
$$
\mathcal{I}^2_\delta[x_0,t_0,u]= \int_{|z|\geq \delta} (u(x_0+z,t_0)-u(x_0,t_0))\mu_{x_0}(dz)\geq 0.
$$
But $u$ attains its maximum at $(x_0,t_0)$ and thus
$
u(x_0+z,t_0)-u(x_0,t_0))\leq 0.
$
Letting $\delta$  go to zero we have
$$
u(z,t_0) = u(x_0,t_0), \hbox{ for all } z\in x_0+\hbox{supp} (\mu_{x_0}).
$$
Arguing by induction, we obtain
$$
u(z,t_0) = u(x_0,t_0), \forall z\in \bigcup_{n\geq 0}A_n. 
$$
Take now $z_0\in\overline{\bigcup_{n\geq 0}A_n}$. Then, there exists a sequence of points $(z_n)_n\subset\bigcup_{n\geq 0}A_n$ converging to $z_0$. Since $u$ is upper semicontinuous, we have
$$
u(z_0,t_0)\geq \limsup_{z_n\rightarrow z_0} u(z_n,t_0) = u(x_0,t_0).
$$
But $(x_0,t_0)$ is a maximum point and the converse inequality holds. Therefore
$$
u(z,t_0) = u(x_0,t_0), \forall z\in\overline{\bigcup_{n\geq 0}A_n}.
$$
\end{proof}

\begin{remark}
In particular when $\hbox{supp}(\mu_x)=\hbox{supp}(\mu)= B$, with $\mu$ being a L\' evy measure and $B$ the unit ball,
$\R^N$ can be covered by translations of $\hbox{supp}(\mu)$ starting at $x_0$
$$
\R^N = x_0 + \bigcup_{n\geq0}\big(\underbrace{\hbox{supp}(\mu)+ ... +\hbox{supp}(\mu)}_{n} \big).
$$ 
and thus $u(\cdot,t_0)$ is constant in $\R^N$.
\end{remark}

\begin{remark}\label{rk::Omega}
Whenever the equation is restricted to $\Omega$, with the corresponding Dirichlet condition outside the domain, then iterations must be taken for all the points in $\Omega$, i.e. $$ A_{n+1} = \bigcup_{x\in\Omega\cap A_{n}}(x + \hbox{supp}(\mu_x))$$
In particular, if $\Omega\subset\overline{\bigcup_{n\geq 0}A_n}$, then $u(\cdot,t_0)$ is constant in $\Omega$.

\begin{figure}
\centering
 \includegraphics[width=8cm]{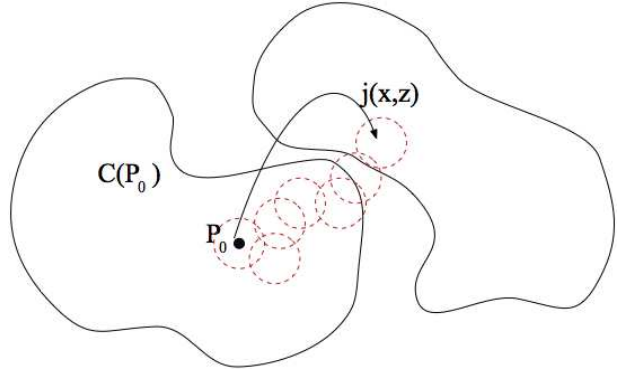}
\caption{Horizontal propagation of maxima by translations of measure supports.}
\label{fig::translations}
\end{figure}

\end{remark}

\begin{remark}
The domain $\Omega$ may not necessarily be connected and  still maxima might propagate, since jumps from one connected component to another might occur when measure supports overlap two or more connected components.
\end{remark}

The previous result has an immediate corollary. If all measure supports have nonempty (topological) interior and contain the origin, strong maximum principle holds.

\begin{corollary}\label{th::int_supp}
Let $\Omega$ be connected, $F$ be as before and $u\in USC(\R^N\times[0,T])$ be a viscosity subsolution of (\ref{ePIDEs}) in  $\Omega\times (0,T)$.
Assume that $\{\mu_x\}_{x\in\Omega}$ satisfies $(M)$ and in addition that the origin belongs to the topological interiors of all measure supports
\begin{equation}\label{int_supp}
 0\in\mathring{\widehat{\hbox{supp}(\mu_x)}},  \forall x\in\Omega.
\end{equation}
If the solution $u$ attains a global maximum at $(x_0,t_0)\in\Omega\times(0,T)$, then $u(\cdot,t_0)$ is constant in the whole domain $\Omega$.
\end{corollary}

\begin{proof}
Consider the iso-level 
\begin{equation*}
\Gamma_{x_0}=\{x\in\Omega; u(x,t_0)=u(x_0,t_0)\}.
\end{equation*}
Then the set is simultaneously open since $ 0\in\mathring{\widehat{\hbox{supp}(\mu_x)}}$ implies, by Theorem \ref{th::trans_supp}, together with Remark \ref{rk::Omega} that for any $x \in \Gamma_{x_0}$ we have 
$$\big (x+ \mathring{\widehat{\hbox{supp}(\mu_x)}}\big) \cap \Omega\subset \Gamma_{x_0}$$ 
and closed because for any $x\in\bar\Gamma_{x_0}$ we have by the upper-semicontinuity of $u$
$$u(x,t_0) \geq \limsup_{y\rightarrow x,\space\ y\in\Gamma_{x_0}} u(y,t_0) = \max_{y\in \Omega} u(y,t_0)  $$
thus $u(x,t_0) = u(x_0,t_0).$
Therefore, $\Gamma_{x_0} = \Omega$ since $\Omega$ is connected and this completes the proof.
\end{proof}

\medskip\subsection[Horizontal Propagation - Nondegeneracy Conditions]
{Horizontal Propagation of Maxima under Nondegeneracy Conditions}
There are cases when conditions (\ref{trans_supp}) and (\ref{int_supp}) fail, 
such as measures whose supports are contained in half space or nonlocal terms acting in one direction, 
as we shall see in section  \S\ref{sec:examples}. 
\medskip

However, we manage to show that, if a viscosity subsolution attains a maximum at $P_0=(x_0,t_0)\in Q_T$, then the maximum propagates in the horizontal component $C(P_0)$, as shown in Figure \ref{fig::smp}. This result is based on \emph{nondegeneracy} $(N)$ and \emph{scaling} $(S)$ properties on the nonlinearity $F$:  

\begin{itemize}
\item [($N$)]  For any $\bar x \in \Omega $ and $0<t_0<T$ there exist $R_0>0$ small enough and $0\leq\eta<1$ such that for any $0 < R < R_0$ and $c >0$
	        $$ F(x,t,p,I-\gamma p\otimes p,\tilde C_\mu - 
	     	c\gamma\int_{\mathcal{C}_{\eta,\gamma}(p)}\big|p\cdot z\big|^2 \mu_{x} (dz))\rightarrow +\infty \hbox{ as } \gamma\rightarrow +\infty$$
	     	uniformly for $ |x- \bar x| \leq R$ and $|t-t_0| \leq R$, $R/2 \leq |p| \leq R$,
		where $$\mathcal{C}_{\eta,\gamma}(p) = \{z; \space\ (1-\eta)|z||p|\leq|p\cdot z|\leq 1/\gamma\}$$
		and $\tilde C_\mu$ appears in $(M)$.  

\item [($S$)] There exist some constants $R_0>0$, $\varepsilon_0>0$  and $\gamma_0>0$
		s.t. for all $0<R<R_0, \varepsilon<\varepsilon_0 \hbox{ and } \gamma\geq\gamma_0$
		the following condition holds for all $ |x- \bar x| \leq R \hbox{ and } |t-t_0| \leq R\hbox{ and } R/2 \leq |p| \leq R$
		\begin{eqnarray}
		 && \nonumber F(x,t,\varepsilon p,\varepsilon (I-\gamma p\otimes p),\varepsilon l)\geq
		 \varepsilon F(x,t,p,I-\gamma p\otimes p,l).
		\end{eqnarray}
\end{itemize} 

As we shall see in \S\ref{sec:examples} the assumption $(M)$ which states that the measure $\mu_x$ is bounded at infinity, uniformly with respect to $x$ and the possible singularity at the origin is of order $|z|^2$ is not sufficient to ensure condition $(N)$. The following assumption is in general needed, provided that the nonlinearity $F$ is nondegenerate in the nonlocal term.

\begin{itemize}
\item [($M^c$)] {For any $x\in\Omega$ there exist $1<\beta<2$, $0\leq\eta<1$ and a constant $C_\mu(\eta)>0$ such that the following holds with $\mathcal C_{\eta,\gamma}(p)$ as before
$$ \int_{\mathcal C_{\eta,\gamma}(p)}|z|^2 \mu_x(dz)\geq C_\mu(\eta)\gamma^{\beta-2}, \forall \gamma\geq 1.$$
}
\end{itemize} 
As pointed out in section \S\ref{sec:examples}, ($M^c$) holds for a wide class of L\'evy measures as well as $(N)-(S)$ for a class of nonlinearities $F$. 
		
\begin{theorem}\label{th::horizprop}
Assume the family of measures $\{\mu_x\}_{x\in\Omega}$ satisfies assumptions $(M)$. Let $u\in USC(\R^N\times[0,T])$ be a viscosity subsolution of (\ref{ePIDEs}) that attains a global maximum at $P_0=(x_0,t_0)\in Q_T$. If $F$ satisfies $(E), (N),$ and $(S)$ then u is constant in $C(P_0)$. 
\end{theorem} \smallskip

\begin{proof}
We proceed as for locally uniformly parabolic equations and argue by contradiction. 
\smallskip

1. Suppose there exists a point $P_1=(x_1, t_0)$ such that $u(P_1)<u(P_0)$. 
The solution $u$ being upper semi-continuous, by classical arguments we can construct for fixed $t_0$ a ball $B(\bar x,R)$ where
$$u(x,t_0)< M=\max_{\R^N}(u(\cdot,t_0)), \forall x\in B(\bar x,R).$$
In addition there exists $x^*\in\partial B(\bar x,R)$ such that $u(x^*,t_0)=M$.
Translating if necessary the center $\bar x$ in the direction $x^*-\bar x$, we can choose $R<R_0$, with $R_0$ given by condition $(N)$.
 
Moreover we can extend the ball to an ellipsoid 
$$\mathcal{E}_R(\bar x, t_0):= \{(x,t);\space\ |x-\bar x|^2 + \lambda |t-t_0|^2< R^2\}$$
with $\lambda$ large enough the function $u$ satisfies
$$
u(x,t)< M, \hbox{ for } (x,t)\in\overline { \mathcal E_R(\bar x, t_0)} \hbox{ s.t. } |x-\bar x| \leq R/2.
$$ 
Remark that  $\ (x^*,t_0)\in\partial\mathcal E_R(\bar x, t_0)$ with $u(x^*,t_0)=M.$
\smallskip

\begin{figure}
\centering
\includegraphics[width=0.8\linewidth]{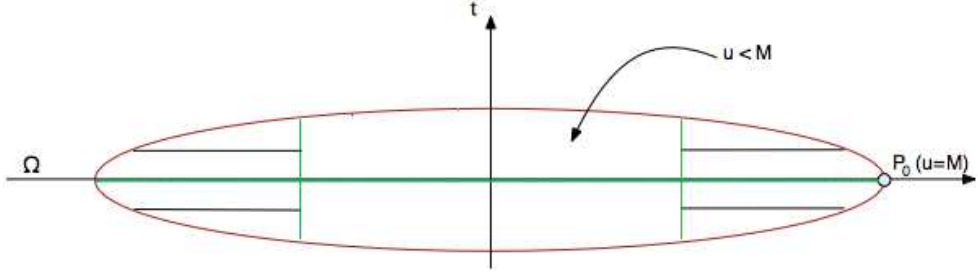}
\caption{\small Construction of the ellipsoid $\mathcal E_R(\bar x, t_0):= \{(x,t);\space\ |x-\bar x|^2 + \lambda |t-t_0|^2< R^2\}$ and of the corresponding auxiliary function $v$ such that inside the dashed area, $v$ is a strict supersolution of the integro-differential equation.}
\end{figure}

2. Introduce the auxiliary function
$$v(x,t)= e^{-\gamma R^2} - e^{-\gamma (|x-\bar x|^2 + \lambda |t-t_0|^2)}$$
where $\gamma>0$ is a large positive constant, yet to be determined.
Note that $ v=0 \hbox{ on }\partial\mathcal E_R(\bar x, t_0) $ and $-1<v<0$, in $\mathcal E_R(\bar x, t_0).$
Denote $d(x,t)=|x-\bar x|^2 + \lambda |t-t_0|^2.$ Direct computations give
\begin{eqnarray*}
       v_t(x,t) & = & 2\gamma e^{-\gamma d(x,t)} \lambda (t-t_0)\\
       Dv(x,t ) & = & 2\gamma e^{-\gamma d(x,t)} (x-\bar x)\\
       D^2v(x,t)& = & 2\gamma e^{-\gamma d(x,t)}(I-2\gamma(x-\bar x)\otimes(x-\bar x)).
\end{eqnarray*}
In upcoming Proposition \ref{prop::NL_est} we show there exist two positive constants $c = c(\eta, R)$ and $\gamma_0>0$ such that for $\gamma\geq\gamma_0$, the following estimate of the nonlocal term holds
$$
\mathcal{I}[x,t,v]\leq 2\gamma e^{-\gamma d(x,t)}\big\{\tilde C_\mu - 
	     c\gamma\int_{\mathcal C_{\eta,\gamma}(x-\bar x)}\big|(x-\bar x) \cdot z\big|^2 \mu_x (dz) \big\}
$$
in the subdomain
$
\mathcal{D}_R(\bar x, t_0):=\{ (x,t)\in\mathcal E_R(\bar x, t_0); |x-\bar x| > R/2\}.
$
\smallskip

3. From the nondegeneracy condition $(N)$ and scaling assumption $(S)$ we get that $v$ is a strict supersolution at points $(x,t)$ in $\mathcal D_R(\bar x, t_0)$. Indeed, for $\gamma$ large enough
$$
F\big(x,t,x-\bar x,I-2\gamma(x-\bar x)\otimes(x-\bar x),
\tilde C_\mu 
	-c\gamma\int_{\mathcal C_{\eta,\gamma}(x-\bar x)}\big|(x-\bar x) \cdot z\big|^2 \mu_x(dz)\big)\big)> 0
$$
On the other hand
\begin{eqnarray*}
&& v_t(x,t) +  F(x,t,Dv(x,t),D^2v(x,t),\mathcal{I}[x,t,v])  \\ 
&& \hspace{2.4cm} = 2 \gamma e^{-\gamma d(x,t)} \lambda (t-t_0)
	+ F\big(x,t,2\gamma e^{-\gamma  d(x,t)}(x-\bar x),\\
&& \hspace{3cm} ...,	2\gamma e^{-\gamma  d(x,t)}(I-2\gamma(x-\bar x)\otimes(x-\bar x)\big), \\   
&& \hspace{3cm} ..., 2\gamma e^{-\gamma  d(x,t)}
	\big\{\tilde C_\mu-c\gamma\int_{\mathcal C_{\eta,\gamma}(x-\bar x)}\big|(x-\bar x) \cdot z\big|^2 \mu_x(dz)\big\})
	\end{eqnarray*}
This further implies that
\begin{eqnarray*}
&& v_t(x,t) +  F(x,t,Dv(x,t),D^2v(x,t),\mathcal{I}[x,t,v])  \\ 
&& \hspace{2.4cm} \geq 2\gamma e^{-\gamma d(x,t)} \big( \lambda (t-t_0) + 
	F\big(x,t,x-\bar x,I-2\gamma(x-\bar x)\otimes(x-\bar x), \\
&& \hspace{3cm} ..., \tilde C_\mu 
	-c\gamma\int_{\mathcal C_{\eta,\gamma}(x-\bar x)}\big|(x-\bar x) \cdot z\big|^2 \mu_x(dz)\big)\big)> 0.
\end{eqnarray*}

Furthermore, the scaling assumption $(S)$ ensures the existence of a constant $\varepsilon_0>0$ such that for all $\varepsilon<\varepsilon_0$, $\varepsilon v$ is a strict supersolution of (\ref{ePIDEs}) in $\mathcal D_R(\bar x,t_0)$. Indeed we have
\begin{eqnarray*}
&&\varepsilon v_t(x,t)+F(x,t,\varepsilon Dv(x,t),\varepsilon D^2v(x,t),\varepsilon \mathcal{I}[x,t,v])  \\
&&   \hspace{1cm}\geq  \varepsilon\big(v_t(x,t)+ F(x,t,Dv(x,t), D^2v(x,t), \mathcal{I}[x,t,v])\big)>0.
\end{eqnarray*}

4. Remark that 
\begin{eqnarray*}
 v \geq 0 & \hbox{ in } & \mathcal E_R^c(\bar x, t_0)\\
 u < M & \hbox{ in } &\mathcal E_R(\bar x,t_0)\setminus\mathcal D_R(\bar x,t_0).
\end{eqnarray*}
Therefore, there exists some $\varepsilon_0>0$ such that for all $\varepsilon<\varepsilon_0$ outside the domain $\mathcal D_R(\bar x,t_0)$ 
$$u(x,t)\leq u(x^*,t_0)+\varepsilon v(x,t).$$
Then we claim that the inequality holds inside $\mathcal D_R(\bar x,t_0)$. Indeed, if $u\leq u(x^*,t_0)+\varepsilon v$ does not hold, then 
$\max_{\mathcal \R^N}(u-M-\varepsilon v)>0$ would be attained in $\mathcal D_R(\bar x, t_0)$ at say, $(x',t')$.
Since  $u$ is a viscosity subsolution the following would hold
$$\varepsilon v_t(x',t')+F(x',t',\varepsilon Dv(x',t'),\varepsilon D^2v(x',t'), \mathcal{I}[x',t',\varepsilon v])\leq 0$$
arriving thus to a contradiction with the fact that $M+\varepsilon v$ is a strict supersolution of (\ref{ePIDEs}). 
\smallskip 

5. The function $u(x,t)-\varepsilon v(x,t)$ has therefore a global maximum at $(x^*,t_0)$. Since $u$ is a viscosity subsolution of (\ref{ePIDEs}), we have
$$\varepsilon v_t(x^*,t_0)+F(x^*,t_0,\varepsilon Dv(x^*,t_0),\varepsilon D^2v(x^*,t_0), \mathcal{I}[x^*,t_0,\varepsilon v])\leq0.$$
As before, we arrived  at a contradiction because $\varepsilon v$ is a strict supersolution and thus the converse inequality holds at $(x^*,t_0)$. Consequently, the assumption made is false and $u$ is constant in the horizontal component of $P_0$.
\end{proof}

In the following we give the estimate for the nonlocal operator acting on the auxiliary function. We use the same notations as before.

\begin{prop}\label{prop::NL_est}
Let $R>0$, $\lambda>0, \gamma>0$ and consider the smooth function 
\begin{eqnarray*}
&& v(x,t)= e^{-\gamma R^2} - e^{-\gamma d(x,t)}\\
&& d(x,t)=|x-\bar x|^2 + \lambda |t-t_0|^2.
\end{eqnarray*}
 Then there exist two constants $c=c(\eta,R)$ and $\gamma_0>0$ such that for $\gamma\geq\gamma_0$ the nonlocal operator satisfies 
\begin{equation*}
\mathcal{I}[x,t,v]\leq 2\gamma e^{-\gamma d(x,t)}\big\{\tilde C_\mu - 
	    c\gamma\int_{\{(1-\eta)|z||x-\bar x|\leq|(x-\bar x) \cdot z|\leq 1/\gamma\}}\big|(x-\bar x) \cdot z\big|^2 \mu_x (dz) \big\}			
\end{equation*}
for all $R/2<|x-\bar x|<R$.
\end{prop}

\begin{proof}
In order to estimate the nonlocal term $\mathcal{I}[x,t,v]$, we split the domain of integration into three pieces and take the integrals on each of these domains. Namely we part the unit ball into the subset 
$$\mathcal{C}_{\eta,\gamma}(x-\bar x) = \{z;(1-\eta)|z||x-\bar x|\leq|(x-\bar x) \cdot z|\leq 1/\gamma\}$$
and its complementary. Indeed $\mathcal{C}_{\eta,\gamma} (x-\bar x)$ lies inside the unit ball, as for $|x-\bar x|\geq R/2$ and for $\gamma$ large enough 
\begin{equation*}
 |z| \leq \frac{1}{\gamma (1-\eta) |x-\bar x|} \leq \frac{2}{\gamma (1-\eta) R} \leq 1.
\end{equation*} 
Thus we write the nonlocal term as the sum 
\begin{eqnarray*}
\mathcal{I}[x,t,v]=\mathcal{T}^1[x,t,v]+\mathcal{T}^2[x,t,v]+\mathcal{T}^3[x,t,v]
\end{eqnarray*}
with
\begin{eqnarray*}
&&  \mathcal{T}^1[x,t,v]=\int_{|z|\geq 1} (v(x+z,t)-v(x,t))\mu_x(dz) \\
&&  \mathcal{T}^2[x,t,v]=\int_{B\setminus\mathcal{C}_{\eta,\gamma}(x-\bar x)}(v(x+z,t)-v(x,t)-Dv(x,t)\cdot z)\mu_x(dz)\\
&&  \mathcal{T}^3[x,t,v]=\int_{\mathcal{C}_{\eta,\gamma}(x-\bar x)}(v(x+z,t)-v(x,t)-Dv(x,t)\cdot z)\mu_x(dz).
\end{eqnarray*}
In the sequel, we show that each integral term is controlled from above by an exponential term of the form $\gamma e^{-\gamma d(x,t)}$. In addition, the last integral is driven by a nonpositive quadratic nonlocal term.

\begin{lemma} \label{est_I1}
We have
\begin{eqnarray*}
&& \mathcal{T}^1[x,t,v]\leq e^{-\gamma d(x,t)}\int_{|z|\geq 1} \mu_x(dz), \forall (x,t)\in\Omega\times[0,T].
\end{eqnarray*}
\end{lemma}

\begin{proof}
The estimate is due to the uniform bound of the measures $\mu_x$ away from the origin. Namely
\begin{eqnarray}
  \nonumber \mathcal{T}^1[x,t,v] &   =   & \int_{|z|\geq 1} (-e^{-\gamma d(x+z,t)} + e^{-\gamma d(x,t)})\mu_x(dz)\\
  \nonumber  \ & \leq  & \int_{|z|\geq 1} e^{-\gamma d(x,t)}\mu_x(dz) = e^{-\gamma d(x,t)} \int_{|z|\geq 1} \mu_x(dz)\leq e^{-\gamma d(x,t)} \tilde C_\mu.
\end{eqnarray}
\end{proof} 

\begin{lemma}\label{est_I2}
We have
 \begin{eqnarray*}
 \mathcal{T}^2[x,t,v] \leq \gamma e^{-\gamma d(x,t)}\int_{B}|z|^2\mu_x(dz), \forall (x,t)\in\Omega\times[0,T].
\end{eqnarray*}
\end{lemma}

\begin{proof}
Note that $\mathcal{T}^2[x,t,v] = - \mathcal{T}^2[x,t,e^{-\gamma d}].$
From Lemma \ref{lemma:exp_ineq1} in Appendix
\begin{eqnarray*}
&& \mathcal{T}^2[x,t,e^{-\gamma d}]\geq e^{-\gamma d(x,t)} \mathcal{T}^2[x,t,-\gamma d]= -\gamma e^{-\gamma d(x,t)} \mathcal{T}^2[x,t, d].
\end{eqnarray*}
Taking into account the expression for $d(x,t)$, we get that
\begin{eqnarray*}
\mathcal{T}^2[x,t,v] & \leq & \gamma e^{-\gamma d(x,t)}\int_{B\setminus\mathcal{C}_{\eta,\gamma}(x-\bar x)} (d(x+z,t) - d(x,t)- Dd(x,t)\cdot z)\mu_x(dz)\\
          \ &   =  & \gamma e^{-\gamma d(x,t)}\int_{B\setminus\mathcal{C}_{\eta,\gamma}(x-\bar x)}|z|^2\mu_x(dz)\\
&\leq &	      \gamma e^{-\gamma d(x,t)}\int_{B}|z|^2\mu_x(dz)\leq\gamma e^{-\gamma d(x,t)} \tilde C_\mu.
\end{eqnarray*}
\end{proof} 

\begin{lemma}\label{est_I3}
There exist two positive constants $c = c(\eta,R)$ and $\gamma_0>0$ such that for $\gamma\geq\gamma_0$
\begin{eqnarray}
\nonumber \mathcal{T}^3[x,t,v]\leq e^{-\gamma d(x,t)}\big(\gamma\int_{B}|z|^2\mu_x(dz) 
- 2 c\gamma^2\int_{\mathcal{C}_{\eta,\gamma}(x-\bar x)}\big|(x-\bar x) \cdot z\big|^2 \mu_x (dz) \big).
\end{eqnarray} 
for all $(x,t)\in\mathcal{D}_R.$
\end{lemma}

\begin{proof}
Rewrite equivalently the integral as
$$
\mathcal{T}^3[x,t,v]=\mathcal{T}^3[x,t,v-e^{-\gamma R^2}] = - \mathcal{T}^3[x,t,e^{-\gamma d}].
$$
We apply then Lemma \ref{lemma:exp_ineq2} in Appendix to the function $e^{-\gamma d}$ and get that for all $\delta>0$ there exists $c=c(\eta,R) >0$ such that  
\begin{eqnarray*}
 \mathcal{T}^3[x,t,e^{-\gamma d}] & \geq & e^{-\gamma d(x,t)}\big(\mathcal{T}^3[x,t,-\gamma d]+
2c \gamma^2\int_{\mathcal{C}_{\eta,\gamma}(x-\bar x)}\big(d(x+z,t)-d(x,t)\big)^2\mu_x (dz)\big)\\
        \ &    = & -\gamma e^{-\gamma d(x,t)}\big(\mathcal{T}^3[x,t,d]-
2c \gamma\int_{\mathcal{C}_{\eta,\gamma}(x-\bar x)}\big(d(x+z,t)-d(x,t)\big)^2\mu_x (dz)\big).
\end{eqnarray*} 
Remark that $\mathcal C_{\eta,\gamma}(x-\bar x) \subseteq D_\delta$ for $\delta = 2 + \frac{2}{(1-\eta)R}$, with
$$\mathcal D_\delta = \{z;\space\ \gamma \big(d(x+z,t)-d(x,t)\big)\leq \delta\}=\{z;\space\ \gamma (2(x-\bar x) \cdot z + |z|^2)\leq \delta\}.$$ 
We have thus
\begin{eqnarray*}
 \mathcal{T}^3[x,t,v]  & \leq & \gamma e^{-\gamma d(x,t)}\big(\mathcal{T}^3[x,t,d]- 2c \gamma\int_{\mathcal{C}_{\eta,\gamma}(x-\bar x)}\big(d(x+z,t)-d(x,t)\big)^2\mu_x (dz)\big).
\end{eqnarray*} 
Taking into account the expression of $d(x,t)$, direct computations give
\begin{eqnarray*}
\mathcal{T}^3[x,t,d]  & = & \int_{\mathcal{C}_{\eta,\gamma}(x-\bar x)} \big(d(x+z,t) - d(x,t)- Dd(x,t)\cdot z\big)\mu_x(dz) \\
		      & = & \int_{\mathcal{C}_{\eta,\gamma}(x-\bar x)}|z|^2\mu_x(dz)\leq \int_{B}|z|^2\mu_x(dz),
\end{eqnarray*} 
while the quadratic term is bounded from below by
\begin{eqnarray*}
\int_{\mathcal{C}_{\eta,\gamma}(x-\bar x)}\big(d(x+z,t)-d(x,t)\big)^2\mu_x (dz)
&  =  & \int_{\mathcal{C}_{\eta,\gamma}(x-\bar x)}\big|2(x-\bar x) \cdot z + |z|^2\big|^2\mu_x (dz)\\
&\geq & \int_{\mathcal{C}_{\eta,\gamma}(x-\bar x)}\big|(x-\bar x) \cdot z\big|^2 \mu_x (dz).
\end{eqnarray*}
Indeed, recall that $|x-\bar x|\geq R/2$ and see that for all $z\in\mathcal{C}_{\eta,\gamma}(x-\bar x)$ 
\begin{eqnarray*}
(1-\eta)|x-\bar x||z|\leq 1/\gamma                    & \Rightarrow & |z| \leq \frac{2}{ \gamma R(1-\eta) }\\
(1-\eta)|x-\bar x||z|\leq |(x-\bar x) \cdot z| & \Rightarrow & |z| \leq \frac{2|(x-\bar x) \cdot z|}{ R(1-\eta) } 
\end{eqnarray*}
Then for $\gamma_0= 4/ R^2(1-\eta)^2$ and $\gamma\geq\gamma_0$ we have the estimate
\begin{eqnarray*}
\big|2(x-\bar x) \cdot z + |z|^2\big| & \geq & 2|(x-\bar x) \cdot z|-|z|^2 
       \geq 2|(x-\bar x) \cdot z| -\frac{4|(x-\bar x) \cdot z|}{\gamma R^2 (1-\eta)^2}\\
\ & = & |(x-\bar x) \cdot z| \big(2-\frac{4}{\gamma R^2 (1-\eta)^2}\big) \geq |(x-\bar x) \cdot z|.
\end{eqnarray*}
Therefore, we obtain the upper bound for the integral term
$$
 \mathcal{T}^3[x,t,v] \leq \gamma e^{-\gamma d(x,t)}\big(\int_{\mathcal{C}_{\eta,\gamma}(x-\bar x)}|z|^2\mu_x(dz) 
- 2c\gamma\int_{\mathcal{C}_{\eta,\gamma}(x-\bar x)}\big|(x-\bar x) \cdot z\big|^2 \mu_x (dz) \big).
$$
\end{proof} 
From the three lemmas estimating the integral terms we deduce that
\begin{eqnarray*}
\mathcal{I}[x,t,v] & \leq & e^{-\gamma d(x,t)}  \Big\{ \int_{|z|\geq 1} \mu_x(dz) + 2\gamma \int_{B} |z|^2\mu_x(dz)
			  - 2c\gamma^2\int_{\mathcal{C}_{\eta,\gamma}(x-\bar x)}\big|(x-\bar x) \cdot z\big|^2 \mu_x (dz)\Big\}\\
		   & \leq & 2\gamma e^{-\gamma d(x,t)} 
			     \big\{ \tilde C_\mu - c\gamma\int_{\mathcal{C}_{\eta,\gamma}(x-\bar x)}\big|(x-\bar x) \cdot z\big|^2 \mu_x (dz)\big\}.
\end{eqnarray*}
\end{proof}
 
\medskip\subsection[Local Vertical Propagation]
{Local Vertical Propagation of Maxima}

We show that if $u\in USC(\R^N\times[0,T])$ is a viscosity subsolution of (\ref{ePIDEs}) which attains a maximum at $P_0=(x_0,t_0)\in Q_T$, then the maximum propagates locally in rectangles, say,
$$\mathcal R (x_0,t_0)= \{(x,t)| |x^i-x^i_0|\leq a^i, t_0-a_0\leq t\leq t_0\}$$
where we have denoted $x = (x^1, x^2, ..., x^N)$.
Denote by $\mathcal R_0 (x_0,t_0)$ the rectangle $\mathcal R (x_0,t_0)$ less the top face $\{t=t_0\}$. 
\smallskip

Local vertical propagation of maxima occurs under softer assumptions on the nondegeneracy and scaling conditions. More precisely, we suppose the following holds: 
\begin{itemize}
\item [($N'$)]  For any $(x_0,t_0)\in Q_T$ there exists $\lambda> 0$  such that
 		$$\lambda+F(x_0,t_0,0,I, \tilde C_\mu)>0$$
 		where $\tilde C_\mu$ is given by assumption $(M)$.
 		
\item [($S'$)] There exist two constants $r_0>0$, $\varepsilon_0>0$ such that for all $\varepsilon<\varepsilon_0$ and $0<r<r_0$ the following condition holds for all 
		$(x,t)\in B((x_0,t_0),r)$, $|p|\leq r$, $l \leq \tilde C_\mu$
		$$F(x,t,\varepsilon p,\varepsilon I, \varepsilon l)\geq	\varepsilon F(x,t, p,I,l).$$ 
\end{itemize} 

\begin{theorem}
Let $u\in USC(\R^N\times[0,T])$ be a viscosity subsolution of (\ref{ePIDEs}) that attains a maximum at $P_0=(x_0,t_0)\in Q_T$. If $F$ satisfies $(E), (N')$ and $(S')$ then for any rectangle $\mathcal R (x_0,t_0)$, $\mathcal R_0 (x_0,t_0)$ contains a point $P\not=P_0$ such that $u(P)=u(P_0)$.
\end{theorem}

\begin{proof}
Similarly to the horizontal propagation of maxima, we argue by contradiction.

1. Suppose there exists a rectangle $\mathcal R(x_0,t_0)$  on which 
$u(x,t)<M=u(x_0,t_0)$, with $\mathcal R_0 (x_0,t_0)\subseteq \Omega\times[0,t_0)$. 
Denote $h(x,t)=\frac12|x-x_0|^2+\lambda(t-t_0)$ with $\lambda>0$ a constant yet to be determined.
Consider the auxiliary function
$$v(x,t)=1-e^{-h(x,t)}.$$
Direct calculations give
$$
\begin{array}{l}
  v_t(x,t) = \lambda e^{- h(x,t)}\\ 
  Dv(x,t)  = e^{- h(x,t)} (x-x_0)\\
  D^2v(x,t)= e^{- h(x,t)}(I-(x-x_0)\otimes(x-x_0)),
\end{array}
$$
Note that 
\begin{eqnarray*}
  v(x_0,t_0)=0 & v_t(x_0,t_0)=\lambda\\
  Dv(x_0,t_0)=0 & D^2v(x_0,t_0)= I.
\end{eqnarray*}
The nonlocal term is written as the sum of two integral operators:
$$\mathcal{I}[x,t,v]=\mathcal{T}^1[x,t,v]+\mathcal{T}^2[x,t,v],$$
where 
\begin{eqnarray*}
 \mathcal{T}^1[x,t,v] & = &\int_{|z|\geq 1} (v(x+z,t)-v(x,t))\mu_x(dz) \\
 \mathcal{T}^2[x,t,v] & = &\int_{B}(v(x+z,t)-v(x,t)-Dv(x,t)\cdot z)\mu_x(dz). 
\end{eqnarray*}
Similarly to Lemma \ref{est_I1} we obtain the estimate:

\begin{lemma} We have
\begin{eqnarray}\nonumber
 \mathcal{T}^1[x,t,v]\leq e^{-h(x,t)}  \int_{|z|\geq 1}\mu_x(dz), \forall (x,t)\in\Omega\times[0,T].
\end{eqnarray}
\end{lemma}

On the other hand, the estimate obtained for the second integral term is softer than the estimate obtained in the case of the horizontal propagation of maxima. 

\begin{lemma}  We have
\begin{eqnarray}\nonumber
 \mathcal{T}^2[x,t,v]\leq e^{-h(x,t)} \int_{B}|z|^2\mu_x(dz), \forall (x,t)\in\Omega\times[0,T].
\end{eqnarray}
\end{lemma}

\begin{proof}
1. From Lemma \ref{lemma:exp_ineq1} we have
$$
\mathcal{T}^2[x,t,v]= -\mathcal{T}^2[x,t,e^{-h}]\leq e^{-h(x,t)}\mathcal{T}^2[x,t,h].
$$
We then use a second-order Taylor expansion for $h$ and get
\begin{eqnarray}
\nonumber \mathcal{T}^2[x,t,h] & = & \frac{1}{2}\int_{B} \sup_{\theta\in(-1,1)}\big(D^2h(x+\theta z, t)z\cdot z\big)\mu_x(dz) \\
\nonumber          \ & = & \frac{1}{2}\int_{B}|z|^2\mu_x(dz)\leq\frac{1}{2}\int_{B}|z|^2\mu_x(dz),
\end{eqnarray}
from where the conclusion.
\end{proof} 
We now go back to the proof of the theorem and see that 
$$
\mathcal{I}[x,t,v]\leq e^{-h(x,t)}\tilde C_\mu.
$$
In particular $\mathcal{I}[x_0,t_0,v]\leq \tilde C_\mu.$   
\smallskip

2. From the nondegeneracy assumption $(N')$ we have that there exists $\lambda>0$ such that
\begin{eqnarray*}
&&  
v_t(x_0,t_0)+F(x_0,t_0,Dv(x_0,t_0),D^2v(x_0,t_0),\mathcal{I}[x_0,t_0,v]) \\
&&  
\hspace{2cm}\geq v_t(x_0,t_0)+F(x_0,t_0,Dv(x_0,t_0),D^2v(x_0,t_0),\tilde C_\mu) \\
&&  
\hspace{2cm} = \lambda+F(x_0,t_0, 0, I, \tilde C_\mu ) > 0.
\end{eqnarray*}
Hence $v$ is a strict supersolution of (\ref{ePIDEs}) at $(x_0,t_0)$. By the continuity of $F$, there exists $r<r_0$ such that $\forall (x,t)\in B((x_0,t_0),r)\subseteq  Q_T$
$$v_t(x,t)+F(x,t,Dv(x,t),D^2v(x,t),\mathcal{I}[x,t,v])\geq C>0.$$
Consider then the set 
$$ S = B((x_0,t_0),r)\cap \{(x,t)| v(x,t)<0\}.$$
By $(S')$ there exists $\varepsilon_0>0$ such that $\forall \varepsilon<\varepsilon_0$, $\varepsilon v$ is a strict supersolution of (\ref{ePIDEs}) in $S$. Indeed
\begin{eqnarray*}
 &&  
\varepsilon v_t(x,t)+F(x,t,\varepsilon Dv(x,t),\varepsilon D^2v(x,t),\varepsilon \mathcal{I}[x,t,v]) \geq \\
 &&  
\hspace{0.4cm}\varepsilon \big( v_t(x,t)+F(x,t,Dv(x,t), D^2v(x,t), \mathcal{I}[x,t,v])\big ) >0.
\end{eqnarray*}

3. Let $\varepsilon_0$ be sufficiently small such that
$$u(x,t)-u(x_0,t_0)\leq\varepsilon v(x,t), \ \forall (x,t)\in\partial S.$$
Then, arguing as in the case of horizontal propagation of maxima we get
$$u(x,t)-u(x_0,t_0)\leq\varepsilon v(x,t), \ \forall (x,t)\in S.$$
Thus $(x_0,t_0)$ is a maximum of $u-\varepsilon v$ with $Dv(x_0,t_0)=\lambda>0$. Since $u$ is a subsolution, we have
$$ \varepsilon v_t(x_0,t_0)+F(x_0,t_0,\varepsilon v(x_0,t_0),\varepsilon Dv(x_0,t_0),\varepsilon D^2v(x_0,t_0),\mathcal{I}[x_0,t_0,\varepsilon v])\leq 0.$$
We arrived at a contradiction with the fact that $\varepsilon v$ is a strict supersolution. Thus, the supposition is false and the rectangle contains a point $P\neq P_0$ such that $u(P)=u(P_0)$.

\end{proof}

\begin{example}\label{ex_dislocation}
Non-local first order Hamilton Jacobi equations describing the dislocation dynamics
\begin{equation*}
 u_t = (c(x) + M[u]) |Du|
\end{equation*}
where $M$ is a zero order nonlocal operator defined by
$$
M[u](x,t) = \int_{\R^N} \big(u(x+z,t) - u(x,t))\mu(dz) 
$$
with $$\mu(dz)= g(\frac{z}{|z|})\frac{dz}{|z|^{N+1}}$$ have vertical propagation of maxima. 
\smallskip

{\rm Indeed, they do not satisfy any of the sets of assumptions required by Theorems \ref{th::trans_supp} and \ref{th::horizprop}.  Particularly nondegeneracy condition $(N)$
$$
-\big(c(x) + \tilde C_\mu\big) |p| > 0
$$  fails for example if $c(x)\geq 0$, and holds whenever $c(x)<-\tilde C_\mu$. Hence, one cannot conclude on horizontal propagation of maxima. 

On the other hand we have local vertical propagation of maxima, since $(N')$ is immediate and $(S')$ is satisfied by $\tilde F = -c(x)|p|$, the linear approximation of the nonlinearity
$$
-\big(c(x) + \varepsilon l) |\varepsilon p| = - \varepsilon c(x) |p| + o(\varepsilon^2).
$$}
\end{example} 

\subsection{Strong Maximum Principle}

When both horizontal and local vertical propagation of maxima occur for a viscosity subsolution of (\ref{ePIDEs}) which attains a global maximum at an interior point,  the function is constant in any rectangle contained in the domain $\overline\Omega\times[0,t_0]$ passing through the maximum point.

\begin{prop}
Let $u\in USC(\R^N\times [0,T])$ be a viscosity subsolution of (\ref{ePIDEs}) in $\Omega\times (0,T)$ that attains a global maximum at $(x_0,t_0)\in Q_T$. Assume the family of measures $\{\mu_x\}_{x\in\Omega}$ satisfies assumption $(M)$ and assume $\Omega\subset\overline{\bigcup_{n\geq 0}A_n}$, with $\{A_n\}_n$ given by (\ref{trans_supp}).
If $F$ satisfies $(E')$, $(S')$ and $(N')$ then u is constant in any rectangle $\mathcal R(x_0,t_0)\subseteq\overline\Omega\times[0,t_0]$.
\end{prop}

\begin{prop}
Let $u\in USC(\R^N\times[0,T])$ be a viscosity subsolution of (\ref{ePIDEs}) that attains a global maximum at $P_0=(x_0,t_0)\in Q_T$. If $F$ satisfies $(E)$, $(N) - (N')$, and $(S)-(S')$, then u is constant in any rectangle $\mathcal R(x_0,t_0)\subseteq\overline\Omega\times[0,t_0]$.
\end{prop}
From the horizontal and local vertical propagation of maxima one can derive the Strong Maximum Principle. The proof is based on geometric arguments and is identical to that for fully nonlinear second order partial differential equations.
\smallskip

\begin{theorem}[Strong Maximum Principle]
Assume the family of measures $\{\mu_x\}_{x\in\Omega}$ satisfies assumption $(M)$. Let $u\in USC(\R^N\times[0,T])$ be a viscosity subsolution of (\ref{ePIDEs}) that attains a global maximum at $P_0=(x_0,t_0)\in Q_T$. If $F$ satisfies $(E)$, $(S) - (S')$, and $(N)-(N')$, then u is constant in $S(P_0)$.
\end{theorem}\label{th:SMaxP_nondeg}
\smallskip

\begin{proof}
Suppose that $u\not\equiv u(P_0)$ in $S(P_0)$. Then there exists a point $Q\in S(P_0)$ such that $u(Q)<u(P_0)$. Then, we can connect $Q$ to $P_0$ by a simple continuous curve $\gamma$ lying in $S(P_0)$ such that the temporal coordinate $t$ us nondecreasing from $Q$ to $P_0$. On the curve $\gamma$ there exists a point $P_1$ take takes the maximum value $u(P_1)= u(P_0)$ and at the same time, for all the points $P$ on $\gamma$ between $Q$ and $P_1$ we have $u(P)<u(P_0)$. We construct a rectangle
$$
x_i^1 - a\leq x_i\leq  x_i^1+a, i=1,n, \ t^1-a<t<t_1
$$
where $(x_i^1, t^1)$ are the coordinates of $P_1$ and $a$ sufficiently small such that the rectangle does not exceed the domain $\Omega$.
Applying the vertical propagation of maxima we deduce that $u\equiv u(P_0)$ in this rectangle. Thus, the function is constant on the arc of the curve lying in this rectangle. But this contradicts the definition of $P_1$.

\end{proof}
\smallskip

Similarly the following holds.

\begin{theorem}[Strong Maximum Principle]
Let $u\in USC(\R^N\times [0,T])$ be a viscosity subsolution of (\ref{ePIDEs}) in $\R^N\times (0,T)$ that attains a global maximum at $(x_0,t_0)\in \R^N\times (0,T]$. Assume the family of measures $\{\mu_x\}_{x\in\Omega}$ satisfies assumption $(M)$ and $F$ satisfies $(E')$, $(S')$ and $(N')$.
Then u is constant in  $\overline{\bigcup_{n\geq 0}A_n}\times[0,t_0]$ with $\{A_n\}_n$ given by (\ref{trans_supp}).
\end{theorem}\label{th:SMAXP_supptrans}

\section{Strong Maximum Principle for L\'evy-It\^o operators}\label{sec:StrongMaxPrinc_LI}

The results established for general nonlocal operators remain true for  L\'evy-It\^o operators. We translate herein the corresponding assumptions and theorems on the Strong Maximum Principle for second order integro-differential equations associated to L\'evy-It\^o operators
$$
\mathcal{J}[x,t,u]=\int_{\R^N} (u(x+j(x,z),t)-u(x,t)-Du(x,t)\cdot j(x,z) 1_B(z))\mu(dz),
$$
where $\mu$ is a L\'evy measure.  
In the sequel we assume  that $F$ respects the scaling assumption $(S)$ and the nondegeneracy condition  
\begin{itemize}
\item [($N_{LI}$)] For any $\bar x \in \Omega $ and $0<t_0<T$ there exist $R_0>0$ small enough and $0<\eta <1$ such that for any $0 < R < R_0$ and $c >0$
	        $$ F(x,t,p,I-\gamma p\otimes p,\tilde C_\mu - 
	     	c\gamma\int_{\mathcal{C}_{\eta,\gamma}(p)}\big|p\cdot j(x,z)\big|^2 \mu (dz))\rightarrow \infty \hbox{ as } \gamma\rightarrow \infty$$
	     	uniformly for $ |x- \bar x| \leq R$ and $|t-t_0| \leq R$, $R/2 \leq |p| \leq R$,
		where $$\mathcal{C}_{\eta,\gamma}(p) = \{z; \space\ (1-\eta)|j(x,z)||p|\leq|p\cdot j(x,z)|\leq 1/\gamma\}.$$ 
\end{itemize} 
and that the L\'evy measure $\mu$ satisfies assumptions
\begin{itemize}
\item [($M_{LI}$)]  there exists a constant $\tilde C_\mu>0$ such that for any $x\in\Omega$, 
$$
\int_B|j(x,z)|^2\mu(dz)+\int_{\R^N\setminus B}\mu(dz)\leq \tilde C_\mu;
$$
\item [($M^c_{LI}$)] {For any $x\in\Omega$ there exist $1<\beta<2$, $0\leq\eta<1$ and a constant $C_\mu(\eta)>0$ such that the following holds
$$ \int_{\mathcal C_{\eta,\gamma}(p)}|j(x,z)|^2 \mu(dz)\geq C_\mu(\eta)\gamma^{\beta-2}, \forall \gamma\geq 1.$$}
\end{itemize} 

Theorem \ref{th::trans_supp} holds for L\'evy-It\^o operators, since L\'evy It\^ o measures can be written as push-forwards of some L\'evy measure $\tilde \mu$
$$ \mu_x = (j(x,\cdot)_*(\tilde\mu)) $$
defined for measurable functions $\phi$ as 
$$ \int_{\R^N}\phi(x)\mu_x(dz) = \int_{\R^N}\phi(j(x,z))\tilde\mu(dz). $$
Hence it is sufficient to replace $\hbox{supp}(\mu_x) = j(x,\hbox{supp}(\tilde\mu))$ in order to get the result.

\begin{theorem}
Assume the L\'evy measure $\mu$ satisfies assumption $(M_{LI})$. Let $u\in USC(\R^N\times [0,T])$ be a viscosity subsolution of (\ref{ePIDEs}) that attains a maximum at $P_0=(x_0,t_0)\in Q_T$. If $F$ satisfies $(E), (S),$ and $(N_{LI})$ then u is constant in $C(P_0)$. 
\end{theorem}\label{horiz_LI}%

\begin{proof}
Since the proof is technically the same, we just point out the main differences, namely the estimate of the nonlocal term.
Consider as before the smooth function 
$$
v(x,t)= e^{-\gamma R^2} - e^{-\gamma d(x,t)}
$$
where $d(x,t)=|x-\bar x|^2 + \lambda |t-t_0|^2,$ for large $\gamma>\gamma_0$. 
Write similarly the nonlocal term as the sum 
\begin{eqnarray*}
\mathcal{J}[x,t,v]=\mathcal{T}^1[x,t,v]+\mathcal{T}^2[x,t,v]+\mathcal{T}^3[x,t,v]
\end{eqnarray*}
where
\begin{eqnarray*}
&&  \mathcal{T}^1[x,t,v]=\int_{|z|\geq 1} (v(x+j(x,z),t)-v(x,t))\mu(dz) \\
&&  \mathcal{T}^2[x,t,v]=\int_{B\setminus\mathcal{C}_{\eta,\gamma}(x-\bar x)}(v(x+j(x,z),t)-v(x,t)-Dv(x,t)\cdot j(x,z))\mu(dz)\\
&&  \mathcal{T}^3[x,t,v]=\int_{\mathcal{C}_{\eta,\gamma}(x-\bar x)}(v(x+j(x,z),t)-v(x,t)-Dv(x,t)\cdot j(x,z))\mu(dz)
\end{eqnarray*}
with $$\mathcal{C}_{\eta,\gamma} (x-\bar x )= \{(1-\eta)|j(x,z)||x-\bar x |\leq|(x-\bar x )\cdot j(x,z)|\leq 1/\gamma\}.$$

\noindent Then the nonlocal operator satisfies for all $(x,t)\in\mathcal D_R$
\begin{eqnarray*}
&& 
\mathcal{T}^1[x,t,v]\leq e^{-\gamma d(x,t)}\int_{|z|\geq 1} \mu(dz).\\
&&
\mathcal{T}^2[x,t,v] \leq \gamma e^{-\gamma d(x,t)}\int_{B}|j(x,z)|^2\mu(dz).\\
&&
\mathcal{T}^3[x,t,v]\leq e^{-\gamma d(x,t)}\big[\gamma\int_{B}|j(x,z)|^2\mu(dz)
- 2 c\gamma^2\int_{\mathcal C_{\eta,_\gamma}(x-\bar x )}\big|(x-\bar x)\cdot j(x,z)\big|^2 \mu (dz) \big].
\end{eqnarray*} 
from where we get the global estimation
\begin{eqnarray*}
\mathcal{J}[x,t,v]   & \leq & e^{-\gamma d(x,t)}\big[\int_{B} \mu(dz)+2\gamma \int_{B}|j(x,z)|^2\mu(dz) \\
&& \hspace{3cm}		- 2 c\gamma^2\int_{\mathcal C_{\eta,_\gamma}(x-\bar x )}\big|(x-\bar x)\cdot j(x,z)\big|^2 \mu (dz) \big]\\
& \leq & 2\gamma e^{-\gamma d(x,t)} \big[\tilde C_\mu 
		-  c\gamma\int_{\mathcal C_{\eta,\gamma}(x-\bar x )}\big|(x-\bar x)\cdot j(x,z)\big|^2 \mu (dz)\big].
\end{eqnarray*} 
\end{proof}

Vertical propagation of maxima holds under the same conditions.
\begin{theorem}
Let $\mu$ be a L\'evy measure satisfying $(M_{LI})$ and $u\in USC(\R^N\times [0,T])$ be a viscosity subsolution of (\ref{ePIDEs}) that attains a maximum at $P_0=(x_0,t_0)\in Q_T$. If $F$ satisfies $(E), (S')$ and $(N')$ then for any rectangle $\mathcal R (x_0,t_0)$, $\mathcal R_0 (x_0,t_0)$ contains a point $P\not=P_0$ such that $u(P)=u(P_0)$.
\end{theorem}\label{vert_LI}
Strong Maximum Principle can thus be formulated for L\'evy-It\^o operators.

\begin{theorem}[Strong Maximum Principle - L\'evy It\^o]
Assume the measure $\mu$ satisfies assumption $(M_{LI})$. Let $u\in USC(\R^N\times [0,T])$ be a viscosity subsolution of (\ref{ePIDEs}) that attains a global maximum at $P_0=(x_0,t_0)\in Q_T$. If $F$ satisfies $(E)$, $(S) - (S')$, and $(N_{LI})-(N')$, then u is constant in $S(P_0)$.
\end{theorem}\label{th:SMaxP_nondeg_LI}

\begin{theorem}[Strong Maximum Principle - L\'evy It\^o]
Let $u\in USC(\R^N\times [0,T])$ be a viscosity subsolution of (\ref{ePIDEs}) in $\R^N\times (0,T)$ that attains a global maximum at $(x_0,t_0)\in \R^N\times (0,T]$. Assume the measure $\mu$ satisfies assumption $(M_{LI})$ and $F$ satisfies $(E_0)$, $(S')$ and $(N')$.
Then u is constant in  $\overline{\bigcup_{n\geq 0}A_n}\times[0,t_0]$ with $\{A_n\}_n$ given by (\ref{trans_supp}).
\end{theorem}\label{th:SMAXP_supptrans_LI}

\section{Examples}\label{sec:examples}
In this  section we discuss the validity of the Strong Maximum Principle on several representative examples.

\medskip\subsection[Horizontal Propagation - Translations of Measure Supports]
{Horizontal Propagation of Maxima by Translations of Measure Supports}
As pointed out in section \ref{sec:StrongMaxPrinc}, translations of measure supports starting at any maximum point $x_0$ lead to horizontal propagation of maxima. 
In particular, Theorem \ref{th::trans_supp} holds for nonlocal terms integrated against L\'evy measures whose supports are the whole space.

\begin{example}\label{ex_thespace}
Consider a pure nonlocal diffusion
\begin{equation}\label{ex_purenonlocal}
u_t - \mathcal{I}[x,t,u]=0 \hbox{ in } \R^N\times(0,T)
\end{equation}
where $\mathcal{I}$ is the L\'evy operator integrated against the L\'evy measure associated with the fractional Laplacian $(-\Delta)^{\beta/2}$:
$$ \mu(dz) = \frac{dz}{|z|^{N+\beta}}. $$
Then \emph{ the support of the measure is the whole space} and thus horizontal propagation of maxima holds for equation (\ref{ex_purenonlocal}) by Theorem \ref{th::trans_supp}.
\end{example}

\begin{example}
Let $N=2$ and consider equation (\ref{ex_purenonlocal}) with $\{\mu_x\}_x$ a family of \emph{L\'evy measures charging two axis} meeting at the origin
$$\mu_x(dz)=1_{\{z_1=\pm\alpha z_2\}}\nu_x(dz),$$
with $\alpha>0$ and $\hbox{supp}(\nu_x)= \R^2,$ for all $x\in\R^2$. Even though zero is not an interior point of the support, translations of measure supports starting at any point $x_0$ cover the whole space, propagating thus maxima all over $\R^2$. 

Similarly, horizontal propagation of maxima holds if \emph{measures charge cones} 
$$\mu_x(dz)=1_{\{|z_1|>\alpha|z_2|\}}\nu_x(dz),$$
with $\alpha>0$ and $\hbox{supp}(\nu_x)= \R^2$.
\end{example}

\medskip\subsection[Strong Maximum Principle - Nondegenerate Nonlocal]
{Strong Maximum Principle driven by the Nonlocal Term under Nondegeneracy Conditions.}
There are equations for which propagation of maxima does not propagate just by translating measure supports, but cases when it requires a different set of assumptions. Nondegeneracy and scaling conditions of the nonlinearity $F$ need to be satisfied in order to have a Strong Maximum Principle. But to ensure condition $(N)$, one has to assume $(M^c)$.

\begin{example}\label{ex_FLhalfspace}
Consider as before equation (\ref{ex_purenonlocal}) and let $\mu$ be the \emph{L\'evy measure} associated to the fractional Laplacian \emph{but restricted to half space}

$$
\mu(dz)= 1_{\{z_1\geq 0\}}(z)\frac{dz}{|z|^{N+\beta}}, \beta\in(1,2).
$$

\noindent where  $z=(z_1,z')\in\R\times\R^{N-1}$. Then $\R^N$ can not be covered by translations of 
the measure support and therefore one cannot conclude the function $u$ is constant on the whole domain, 
except for particular cases like the periodic case. However, $C^{0,\alpha}$ regularity results hold 
(cf. \cite{BCI:11:HdrNL}) and we expect to have Strong Maximum Principle.\smallskip

\begin{figure}[ht]
\centering\includegraphics[width=8cm]{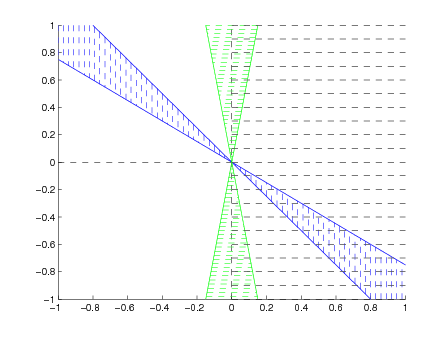}
\caption{Even if the measures are defined on half space, we can always find half cones where the integral terms are nondegenerate.}
\end{figure} 

{\rm We show that the nondegeneracy and scaling assumptions are satisfied in the case of
 Example \ref{ex_FLhalfspace}. Before proceeding to the computations, remark that 

\begin{eqnarray*}
\mathcal{C}_{\eta,\gamma}(p)& = & \{z; \space (1-\eta)|p||z|\leq|p\cdot z|\leq 1/\gamma\} \\
& = &  \{ z; \space (1-\eta)|p||\gamma z|\leq|p\cdot \gamma z|\leq 1\} = \gamma^{-1}\mathcal C_{\eta,1}(p)
\end{eqnarray*}
\begin{eqnarray*} 
\gamma\int_{\mathcal{C}_{\eta,\gamma}(p)\cap\{z_1\geq 0\}} |p\cdot z|^2 \frac{dz}{|z|^{N+\beta}} & = &
 \gamma^{-1}\int_{\mathcal{C}_{\eta,1} (p)\cap\{z_1\geq 0\}} \gamma^{\beta} |p\cdot z|^2 \frac{dz}{|z|^{N+\beta}} \\
 \ &\geq & \gamma^{\beta-1}  |p|^2(1-\eta)^2\int_{\mathcal{C}_{\eta,1}(p) \cap\{z_1\geq 0\}}|z|^2 \frac{dz}{|z|^{N+\beta}} \\
 \ &  =  & C(\eta)\gamma^{\beta-1} |p|^2 \,
\end{eqnarray*}
where $C(\eta) = (1-\eta)^2\int_{\mathcal{C}_{\eta,1}(p) \cap\{z_1\geq 0\}}|z|^2 dz/|z|^{N+\beta}$ is a positive constant. 

This further implies nondegeneracy condition (N). Indeed, there exist $R_0>0$ small enough and $0\leq\eta <1$ such that for any $0 < R < R_0$ and for all  $R/2<|p|<R$ 
\begin{eqnarray*}
-\tilde C_\mu  + c\gamma\int_{\mathcal{C}_{\eta,\gamma}(p)\cap\{z_1\geq 0\}} |p\cdot z|^2 \frac{dz}{|z|^{N+\beta}} 
	 \geq   -\tilde C_\mu + \tilde C(\eta)\gamma^{\beta-1} |p|^2 \rightarrow \infty \hbox{ as } \gamma\rightarrow \infty 
\end{eqnarray*}
as long as $\beta>1$. The rest of assumptions follow immediately.

Similar results hold for the following PIDE arising in the context of growing interfaces \cite{Woyczy:01:Levy}:
\begin{equation}\label{ex_growint}
 u_t + \frac12|Du|^2 - \mathcal{I}[x,t,u] = 0, \hbox{ in } \R^N\times(0,T)
\end{equation}
with $\mathcal{I}$ is a general nonlocal operator of form (\ref{NL_op}). 
}
\end{example}

\begin{remark}
For integro-differential equations of the type
\begin{equation}
 u_t  + b(x,t) |Du|^{m} - \mathcal I[x,t,u]=0 \hbox{ in } \R^N\times(0,T)
\end{equation}
with $b$ a continuous function and $\mu$ as in Example \ref{ex_FLhalfspace}. Strong Maximum Principle holds 
for $m\geq 1$, and for $m< 1$ if $b(\cdot)\geq0$. 
\end{remark}

\medskip\subsection{Strong Maximum Principle coming from Local Diffusion Terms}
Theorem \ref{th:SMaxP_nondeg} applies to integro-differential equations uniformly elliptic with respect to the diffusion term and linear in the nonlocal operator. 

\begin{example}
Quasilinear parabolic integro-differential equations of the form
\begin{equation}\label{ex_DwN}
 u_t - \hbox{tr}(A(x,t)D^2u) - \mathcal{I}[x,t,u] = 0 \hbox{ in } \R^N\times(0,T)
\end{equation}
with $A(x,t)$ such that
$$ a_0(x,t) I\leq A(x,t)\leq a_1(x,t) I, \space\ a_1(x,t)\geq a_0(x,t)>0$$ 
satisfy Strong Maximum Principle. 

{\rm We check the nondegeneracy and scaling conditions for this equation.
\begin{eqnarray*} 
(N)
&& -\hbox{trace}(A(x,t)(I-\gamma p\otimes p)) -\tilde C_\mu + c\gamma\int_{\mathcal{C}_\gamma} |p\cdot z|^2 \mu_x(dz) = \\
&& -\hbox{trace}(A(x,t))+\gamma\hbox{trace}(A(x,t)p\otimes p))-\tilde C_\mu + c\gamma\int_{\mathcal{C}_\gamma} |p\cdot z|^2 \mu_x(dz) \geq  \\
&& \underbrace{ - a_1(x,t) N + a_0(x,t)\gamma |p|^2 -\tilde C_\mu}_{\gg 0, \hbox{ for $\gamma$ large}} + \underbrace{c\gamma\int_{\mathcal{C}_\gamma} |p\cdot z|^2 \mu_x(dz)}_{\geq 0}.\\
(N')
&& \lambda - \hbox{trace}(A(x,t)) -\tilde C_\mu \geq \lambda  - a_1(x,t) N -\tilde C_\mu >0.
\end{eqnarray*} 
The scaling properties are immediate since the nonlinearity  is 1-homogeneous.}
\end{example}

\begin{remark}
More generally, one can consider equations of the form
\begin{equation}\label{ex_DwNG}
u_t + F(x,t,Du,D^2u) -\mathcal I[x,t,u] = 0
\end{equation}
for which the corresponding differential operator $F$ satisfies the nondegeneracy and scaling assumptions. 
The nonlocal term is driven by the second order derivatives and thus Strong Maximum Principle holds.
\end{remark}

\medskip\subsection{Strong Maximum Principle for Mixed Differential-Nonlocal terms}
We consider mixed integro-differential equations, i.e. equations for which local diffusions occur only in certain directions and nonlocal diffusions on the orthogonal ones, and show they satisfy Strong Maximum Principle. This is quite interesting, as the equations might be degenerate in both local or nonlocal terms, but the overall behavior is driven by their interaction (the two diffusions cannot cancel simultaneously). 

\begin{figure}[ht]
\centering
\includegraphics[width = 0.5\linewidth]{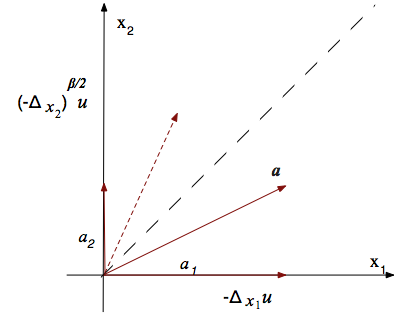}
\caption{\small Local diffusions occur only in $x_1$-directions and fractional diffusions in $x_2$-directions.}
\end{figure}

\begin{example}
Consider the following equation where local and nonlocal diffusions are mixed up 
\begin{equation}
 u_t - \mathcal{I}_{x_1}[u] -\Delta_{x_2} u = 0 \hbox{ in } \R^N\times (0,T)
\end{equation}
for $x=(x_1,x_2)\in\R^d\times\R^{N-d}$. 
The diffusion term gives the ellipticity in the direction of $x_2$, while the nonlocal term gives it in the direction of $x_1$ 
$$
\mathcal{I}_{x_1}[u] = \int_{\R^d} (u(x_1+z_1,x_2)-u(x)-D_{x_1}u(x)\cdot z_1 1_{B}(z_1))\mu_{x_1}(dz_1)
$$
where $\mu_{x_1}$ is a L\'evy measure satisfying $(M)$ with $\tilde C^1_\mu$.
The payoff for the Strong Maximum Principle to hold is assumption $(M^c)$, with $\beta>1$; 
then Theorem \ref{th:SMaxP_nondeg} applies.\smallskip

{\rm Indeed the nondegeneracy conditions $(N)$ and $(N')$ hold, because  when $\gamma$ is large enough and $\beta>1$ the following holds
\begin{eqnarray*}
(N)
&& - I_{N - d} + \gamma p_2 \otimes p_2 -\tilde C^1_\mu + c \gamma \int_{\mathcal C^1_{\eta,\gamma}(p_1)} |p_1\cdot z_1|^2 \mu_{x_1}(dz_1) \geq \\
&& - (N-d)+ \gamma |p_1|^2 - \tilde C^1_\mu + c\gamma (1-\eta)^2 |p_1|^2 \int_{\mathcal C^1_{\eta,\gamma}(p_1)}|z_1|^2 \mu_{x_1}(dz_1) \geq \\
&& - (N-d+\tilde C^1_\mu) + \gamma |p_1|^2 + \tilde C^1(\eta) \gamma^{\beta -1}|p_1|^2 \geq  - c_0 + c_1 \gamma^{\beta -1}\big(|p_1|^2+|p_2|^2\big) \\
\end{eqnarray*}
where $\tilde C^1(\eta)$, $c_0$ and $c_1$ are positive constants and 
$$\mathcal C^1_{\eta,\gamma}(p_1) = \{z_1\in \R^d; \space (1-\eta)|p_1||z_1|\leq|p_1\cdot z_1|\leq 1/\gamma\}.$$
As far as the scaling assumptions are concerned it is sufficient to see that the nonlinearity is 1-homogeneous.}
\end{example}

\begin{remark}
In general, linear integro-differential equations of the form
 \begin{equation}
 u_t - a(x)\mathcal{I}_{x_1}[u] - c(x) \Delta_{x_2} u = 0 \hbox{ in } \R^N \times (0,T)
\end{equation}
or 
\begin{equation}
 u_t -a(x)\mathcal{I}_{x_1}[u] - c(x)\mathcal{I}_{x_2}[u] = 0  \hbox{ in } \R^N \times (0,T)
\end{equation}
satisfy Strong Maximum Principle if the corresponding L\'evy measure(s) 
verify $(M)$ and $(M^c)$, with $\beta>1$ and if $a,c \geq\zeta>0$ in $\R^N$.\smallskip

{\rm Indeed, $F$ is 1-homogeneous and (N) holds:
\begin{eqnarray*}
 c(x)\big(- I_{N - d} + \gamma p_2 \otimes p_2\big) + 
 a(x)\big(-\tilde C^1_\mu + c \gamma \int_{\mathcal  C_{\eta,\gamma}(p_1)} |p_1\cdot z_1|^2 \mu_{x_1}(dz_1)\big) \geq \\
  \geq  - c_0(a(x) + c(x)) + c_1 \gamma^{\beta -1}\big(a(x)|p_1|^2 + c(x)|p_2|^2\big) 
\end{eqnarray*}
respectively
\begin{eqnarray*}
&& a(x)\big(-\tilde C^1_\mu + c \gamma \int_{\mathcal C_{\eta,\gamma}(p_1)}  |p_1\cdot z_1|^2 \mu_{x_1}(dz_1)\big) + \\
&& c(x)\big(-\tilde C^2_\mu + c \gamma \int_{\mathcal C_{\eta,\gamma}(p_2)}  |p_2\cdot z_2|^2 \mu_{x_2}(dz_2)\big) \geq \\
&& \hspace{1cm} \geq  - c_0(a(x) + c(x)) + c_1 \gamma^{\beta -1}\big(a(x)|p_1|^2 + c(x)|p_2|^2\big) .
\end{eqnarray*}
where $\mathcal C_{\eta,\gamma}(p_i) = \{z_i; \space |p_i\cdot z_i|\leq 1/\gamma\}$, for $i=1,2$.}
\end{remark}

\bigskip\section{Strong Comparison Principle}\label{sec:StrongCompPrinc}

Let $\Omega\subset\R^N$ be a bounded, \emph{connected} domain. In this section, we use Strong Maximum Principle to prove a Strong Comparison Result of viscosity sub and supersolution for integro-differential equations of the form (\ref{ePIDEs})
\begin{equation}\label{comp_eq}
 u_t + F(x,t,Du,D^2u,\mathcal{J}[x,t,u]) =0, \hbox{ in } \Omega\times (0,T)
\end{equation}
with the Dirichlet boundary condition
\begin{equation}\label{comp_bd}
 u = \varphi \hbox{ on }  \Omega^c\times [0,T]
\end{equation}
where $\varphi$ is a continuous function.  

Let $\mu$ be a L\'evy measure satisfying $(M_{LI})$.
Assume that the function $j$ appearing in the definition of $\mathcal J$ has the following property: there exists $C_0>0$ such that for all $x,y\in\Omega$ and $|z|\leq\delta$
\begin{eqnarray*}
 |j(x,z)|\leq C_0 |z| \\
 |j(x,z) - j(y,z)|\leq C_0 |z| |x-y|.
\end{eqnarray*}

We will need some additional assumptions on the equation, that we state in the following. Suppose the nonlinearity $F$ is Lipschitz continuous with respect to the variables $p$, $X$ and $l$ and for each $0<R<\infty$ there exist a function $\omega_R(r)\rightarrow 0 $, as $r\rightarrow 0$, $c_R$ a positive constant and $0\leq\lambda_R<\Lambda_R$ such that 

\begin{eqnarray*}(H)
&& F(y,s,q,Y,l_2)-F(x,t,p,X,l_1) \leq \\
&& \hspace{1cm} \omega_R(|(x,t)-(y,s)|) + c_R|p-q| + \textcolor{black}{\mathcal M_R^+}(X-Y) + c_R(l_1 - l_2),
\end{eqnarray*}
for all $x,y\in\Omega$, $t,s\in[0,T]$, $X,Y\in\mathbb S^N(\Omega)$ satisfying for some $\varepsilon>0$

\begin{equation*}
\begin{bmatrix} 
X &  0 \\
0 & -Y 
\end{bmatrix} 
\leq \frac{1}{\varepsilon}
\begin{bmatrix} 
 I & -I \\
-I &  I 
\end{bmatrix}
+
\begin{bmatrix} 
 Z & 0\\
 0 & 0 
\end{bmatrix}
, \hbox{ with } Z\in \mathbb S^N (\Omega)
\end{equation*}
and $|p|,|q|\leq R$ and $l_1,l_2\in\R$, where $\textcolor{black}{\mathcal M_R^+}$ is Pucci's \textcolor{black}{maximal} operator:
$$\textcolor{black}{\mathcal M_R^+}(X) = \Lambda_R\sum_{ \lambda_j>0} \lambda_j + \lambda_R \sum_{ \lambda_j<0} \lambda_j$$ 
with $\lambda_j$ being  the eigenvalues of $X$.  

\begin{theorem}[Strong Comparison Principle]
Assume the L\'evy measure $\mu$ satisfies assumption $(M_c)$ with $\beta>1$. Let $u\in USC(\R^N\times [0,T])$ be a viscosity subsolution and $v\in LSC(\R^N\times [0,T])$  a viscosity supersolution of (\ref{ePIDEs}), with the Dirichlet boundary condition (\ref{comp_bd}). Suppose one of the following conditions holds:
 \begin{itemize*}
  \item[(a)]  $F$ satisfies $(H)$ with $w_R$ and $c_R$ independent of $R$ or
  \item[(b)]  $u(\cdot,t), v(\cdot,t)\in Lip(\Omega)$, $\forall t\in[0,T)$ and $F$ satisfies $(H)$.
 \end{itemize*}
If $ u-v$ attains a maximum at $ P_0=(x_0,t_0) \in \Omega \times (0,T)$, then $u-v$ is constant in $C(P_0)$.
\end{theorem} 

\begin{proof}
The proof relies on finding the equation for which $w = u-v \in USC(\R^N\times [0,T])$ 
is a viscosity subsolution and applying strong maximum principle results for the latter. 
However, the conclusion is not immediate as linearizion does not go hand in hand with the 
viscosity solution theory approach and difficulties imposed by the behavior of the measure near the singularity might appear.\smallskip

1. Let $w = u-v$ and consider $\phi$ a smooth test-function such that $w-\phi$ has a strict global maximum at $(x_0,t_0)$. We penalize the test function around the maximum point, by doubling the variables, i.e. we consider the auxiliary function 
$$\Psi_{\varepsilon,\eta}(x,y,t,s) = u(x,t)-v(y,s)- \frac{|x-y|^2}{\varepsilon^2} - \frac{(t-s)^2}{\eta^2} - \phi(x,t).$$
Then there exist a sequence of global maximum points $(x_\varepsilon,y_\varepsilon,t_\eta,s_\eta)$ of function $\Psi_{\varepsilon,\eta}$ with the properties
\begin{eqnarray*}
&& (x_\varepsilon,t_\eta),(y_\varepsilon,s_\eta)\rightarrow (x_0,t_0) \hbox{ as }\eta,\varepsilon\rightarrow 0\\
&& \frac{|x_\varepsilon-y_\varepsilon|^2}{\varepsilon^2}\rightarrow \varepsilon \hbox{ as } \varepsilon\rightarrow 0\\
&& \frac{(t_\eta-s_\eta)^2}{\eta^2}\rightarrow 0 \hbox{ as } \eta\rightarrow 0
\end{eqnarray*}
and the test-function $\varphi$ being continuous
\begin{equation}\label{conv_max}
\lim_{\eta,\varepsilon \rightarrow 0}(u(x_\varepsilon,t_\eta) -v(y_\varepsilon,s_\eta))=u(x_0,t_0) - v(x_0,t_0).
\end{equation}
In addition, there exist $X_\varepsilon,Y_\varepsilon\in\mathbb S^N$ such that 
\begin{eqnarray*}
 (a_\eta + \phi_t(x_\varepsilon,t_\eta), p_\varepsilon + D\phi(x_\varepsilon,t_\eta), X_\varepsilon + D^2\phi(x_\varepsilon,t_\eta)) 
 \in \overline{\mathcal D}^{2,+} u(x_\varepsilon,t_\eta)\\
 (a_\eta,p_\varepsilon,Y_\varepsilon) \in \overline{\mathcal D}^{2,-}v(y_\varepsilon,s_\eta)
\end{eqnarray*}
\begin{equation*}
\begin{bmatrix} 
X_\varepsilon + D\phi(x_\varepsilon,t_\eta)&  0 \\
0 & -Y_\varepsilon 
\end{bmatrix} 
\leq \frac{4}{\varepsilon^2}
\begin{bmatrix} 
 I & -I \\
-I &  I 
\end{bmatrix}
+
\begin{bmatrix} 
D\phi(x_\varepsilon,t_\eta)&  0 \\
0 & 0
\end{bmatrix}
\end{equation*}
and $p_\varepsilon$, $ a_\eta$ are defined by
\begin{equation*}
p_\varepsilon:= 2\frac{x_\varepsilon-y_\varepsilon}{\varepsilon^2} \hbox{ and }
a_\eta: = 2\frac{t_\eta-s_\eta}{\eta^2}.
\end{equation*}
Consider the test function 
$$\phi^1_{\varepsilon,\eta}(x,t) = v(y_\varepsilon,s_\eta) + \frac{|x-y_\varepsilon|^2}{\varepsilon^2} + \frac{(t-s_\eta)^2}{\eta^2} + \phi(x,t).$$
Then $u-\phi^1_{\varepsilon,\eta}$ has a global maximum at $(x_\varepsilon,t_\eta)$. But $u$ is a subsolution of (\ref{ePIDEs}) and thus for $\delta>0$ the following holds
\begin{eqnarray*}
&& \phi_t(x_\varepsilon,t_\eta) + a_\eta
      + F\big(x_\varepsilon,t_\eta, D\phi(x_\varepsilon,t_\eta) + p_\varepsilon, D^2\phi(x_\varepsilon,t_\eta) + X_\varepsilon, ...\\
&& \hspace{2cm}  ..., 
		       \mathcal J^1_\delta[x_\varepsilon,t_\eta,\phi + \frac{|x - y_\varepsilon|^2}{\varepsilon^2}] 
		       + \mathcal J^2_\delta[x_\varepsilon,t_\eta, D\phi(x_\varepsilon,t_\eta)+ p_\varepsilon,u]\big)\leq 0.
\end{eqnarray*}
Similarly, consider the test function 
$$\phi^2_{\varepsilon,\eta}(y,s) = u(x_\varepsilon,t_\eta) - \frac{|x_\varepsilon-y|^2}{\varepsilon^2} - \frac{(t_\eta-s)^2}{\eta^2} - \phi(x_\varepsilon,t_\eta).$$
Then $v-\phi^2_{\varepsilon,\eta}$ has a global minimum at $(y_\varepsilon,s_\eta)$. But $v$ is a supersolution of (\ref{ePIDEs}) and thus:
\begin{eqnarray*}
&& a_\eta + F\big (y_\varepsilon,s_\eta, p_\varepsilon, 
		       Y_\varepsilon, \mathcal J^1_\delta[y_\varepsilon,s_\eta,\frac{|x_\varepsilon-y|^2}{\varepsilon^2}]
		       +\mathcal J^2_\delta [y_\varepsilon,s_\eta, p_\varepsilon,v]\big)\geq 0.
\end{eqnarray*}
Subtracting the two inequalities and taking into account $(H)$ we get that for all $\delta >0$ 
\begin{eqnarray*}
 \phi_t(x_\varepsilon,t_\eta) 
 & - & \omega(|(x_\varepsilon,t_\eta) - (y_\varepsilon,s_\eta)|)  
	- c|D\phi(x_\varepsilon,t_\eta)| - \textcolor{black}{\mathcal M^+}(D^2\phi(x_\varepsilon,t_\eta) + X_\varepsilon - Y_\varepsilon) \\ 
 & - & c(\mathcal J^1_\delta[x_\varepsilon,t_\eta,\phi + \frac{|x - y_\varepsilon|^2}{\varepsilon^2}] 
		       + \mathcal J^2_\delta[x_\varepsilon,t_\eta,D\phi(x_\varepsilon,t_\eta)+ p_\varepsilon,u] )\\
 & - & c(\mathcal J^1_\delta[y_\varepsilon,s_\eta,-\frac{|x_\varepsilon-y|^2}{\varepsilon^2}]
		       -\mathcal J^2_\delta [y_\varepsilon,s_\eta, p_\varepsilon,v]) \leq0.
\end{eqnarray*}
Taking into account the matrix inequality and the sublinearity of Pucci's operator, we deduce that 
$$
\textcolor{black}{\mathcal M^+}(D^2\phi(x_\varepsilon,t_\eta) + X_\varepsilon - Y_\varepsilon) \leq \textcolor{black}{\mathcal M^+}(D^2\phi(x_\varepsilon,t_\eta) ).
$$
On the other hand, we seek to estimate the integral terms. For this purpose denote 
\begin{eqnarray*}
&& l_u(z):= u(x_\varepsilon+j(x_\varepsilon,z),t_\eta) - u(x_\varepsilon,t_\eta) - (p_\varepsilon+D\phi(x_\varepsilon,t_\eta))\cdot j(x_\varepsilon,z) \\
&& l_v(z):= v(y_\varepsilon+j(y_\varepsilon,z),s_\eta) - v(y_\varepsilon,s_\eta) - p_\varepsilon \cdot j(y_\varepsilon,z) \\
&& l_\phi(z):= \phi(x_\varepsilon+j(x_\varepsilon,z),t_\eta) - \phi(x_\varepsilon,t_\eta) - D\phi(x_\varepsilon,t_\eta) \cdot j(x_\varepsilon,z).
\end{eqnarray*}
Fix $\delta' \gg \delta$ and split the integrals into:
\begin{eqnarray*}
\mathcal J^2_\delta[x_\varepsilon,t_\eta,p_\varepsilon+D\phi(x_\varepsilon,t_\eta), u] 
&=& \mathcal J^2_{\delta'}[x_\varepsilon,t_\eta,p_\varepsilon+D\phi(x_\varepsilon,t_\eta),u] +
\int_{\delta<|z|<\delta'} l_u(z)\mu(dz)\\
\mathcal J^2_\delta[y_\varepsilon,s_\eta,p_\varepsilon ,v] 
&=& \mathcal J^2_{\delta'}[y_\varepsilon,s_\eta,p_\varepsilon ,v] + \int_{\delta<|z|<\delta'}l_v(z)\mu(dz).
\end{eqnarray*}
Since $(x_\varepsilon,y_\varepsilon,t_\eta,s_\eta)$ is a maximum of $\Psi_{\varepsilon,\eta}$ we have
\begin{eqnarray*}
u(x_\varepsilon+j(x_\varepsilon,z),t_\eta)-v(y_\varepsilon+ j(y_\varepsilon,z),s_\eta)
- \frac{|x_\varepsilon+j(x_\varepsilon,z)-y_\varepsilon- j(y_\varepsilon,z)|^2}{\varepsilon^2} -\\
- \phi(x_\varepsilon+j(x_\varepsilon,z),t_\eta) \leq u(x_\varepsilon,t_\eta)-v(y_\varepsilon,s_\eta)
- \frac{|x_\varepsilon-y_\varepsilon|^2}{\varepsilon^2} - \phi(x_\varepsilon,t_\eta)
\end{eqnarray*}
from where we get
\begin{eqnarray*}
l_u(z) - l_v(z) & \leq & l_\phi(z)+\frac{|j(x_\varepsilon,z)-j(y_\varepsilon,z)|^2}{\varepsilon^2} \\
\ & \leq  & l_\phi(z)+C_0^2 \frac{|x_\varepsilon - y_\varepsilon|^2}{\varepsilon^2}|z|^2 .
\end{eqnarray*}
This leads us to 
$$
\int_{\delta<|z|<\delta'} l_u(z)\mu(dz) - \int_{\delta<|z|<\delta'} l_v(z)\mu(dz) 
\leq \int_{\delta<|z|<\delta'} l_\phi(z)\mu(dz) + O(\frac{|x_\varepsilon - y_\varepsilon|^2}{\varepsilon^2}).
$$
Letting first $\delta$ go to zero, we get 
\begin{eqnarray*}
&& \limsup_{\delta\rightarrow 0}\big(\mathcal J^2_\delta   [x_\varepsilon,t_\eta,p_\varepsilon+D\phi(x_\varepsilon,t_\eta),u] 
 - \mathcal J^2_\delta[y_\varepsilon,s_\eta,p_\varepsilon, v] \big)\leq \\ 
&& \hspace{2cm}\leq\mathcal J^2_{\delta'}[x_\varepsilon,t_\eta,p_\varepsilon+D\phi(x_\varepsilon,t_\eta),u] 
- \mathcal J^2_{\delta'}[y_\varepsilon,s_\eta,p_\varepsilon,v] \\
& & \hspace{2cm}
 + \mathcal J^1_{\delta'} [x_\varepsilon,t_\eta,\phi] + O(\frac{|x_\varepsilon - y_\varepsilon|^2}{\varepsilon^2})
\end{eqnarray*}
whereas close to the origin
\begin{eqnarray*}
\mathcal J^1_\delta[x_\varepsilon,t_\eta,\frac{|x - y_\varepsilon|^2}{\varepsilon^2}] 
- \mathcal J^1_\delta[y_\varepsilon,s_\eta,-\frac{|x_\varepsilon-y|^2}{\varepsilon^2}] 
= \frac{2}{\varepsilon^2} \int_{|z|\leq \delta} |j(x_\varepsilon,z)|^2\mu(dz)\rightarrow 0 \\
\mathcal J^1_\delta[x_\varepsilon,t_\eta,\phi] \leq  \int_{|z|\leq \delta} \big(\sup_{|\theta|<1} D^2\phi(x_\varepsilon + \theta j(x_\varepsilon,z),t_\eta)j(x_\varepsilon,z)\cdot j(x_\varepsilon,z)\big) \mu(dz)\rightarrow 0.
\end{eqnarray*}
Furthermore, employing (\ref{conv_max}) and the regularity of the test function $\phi$, as well as the upper semicontinuity of $u-v$ and the continuity of the jump function $j$, we have
\begin{eqnarray*}
&& \limsup_{\eta,\varepsilon\rightarrow0} \big(\mathcal J^2_{\delta'}[x_\varepsilon,t_\eta,p_\varepsilon+D\phi(x_\varepsilon,t_\eta),u] 
 - \mathcal J^2_{\delta'}[y_\varepsilon,s_\eta,p_\varepsilon,v] \big) \\
&&  \hspace{1cm}\leq 
   \int_{|z|\geq \delta'} \limsup_{\eta,\varepsilon\rightarrow0} \big((u(x_\varepsilon+j(x_\varepsilon, z),t_\eta) - v(y_\epsilon+j(y_\varepsilon,z),s_\eta))\\
&& \hspace{3cm}
- (u(x_\varepsilon,t_\eta) - v(y_\varepsilon,s_\eta))\\
&& \hspace{3cm}   
   -( D\phi(x_\varepsilon,t_\eta)\cdot j(x_\varepsilon,z) + p_\varepsilon \cdot (j(x_\varepsilon,z)-j(y_\varepsilon,z))) 1_B(z)\big)\mu(dz) \\
&& \hspace{1cm}   
   \leq \int_{|z|\geq \delta'}\big(\limsup_{\eta,\varepsilon\rightarrow0} (u(x_\varepsilon+j(x_\varepsilon, z),t_\eta) - v(y_\varepsilon+j(y_\varepsilon,z),s_\eta))\\
&& \hspace{3cm}     
   -\lim_{\eta,\varepsilon\rightarrow0}(u(x_\varepsilon,t_\eta) - v(y_\varepsilon,s_\eta))\\
&& \hspace{3cm}        
   -\lim_{\eta,\varepsilon\rightarrow0} D\phi(x_\varepsilon,t_\eta) \cdot j(x_\varepsilon,z) 1_B(z)\big)\mu(dz) \\
&&  \hspace{1cm}\leq 
   \int_{|z|\geq \delta'} \big((u(x_0+j(x_0,z),t_\eta) - v(x_0+j(x_0,z),t_0)) \\
&& \hspace{3cm} 
    - (u(x_0,t_0) - v(x_0,t_0))\\
&& \hspace{3cm} 
   - D\phi(x_0,t_0)\cdot j(x_0,z) \big)\mu(dz) = \mathcal J^2_{\delta'}[x_0,t_0,D\varphi(x_0,t_0), w].
\end{eqnarray*}
Passing to the limits in the viscosity inequality we get, for all $\delta'>0$ that 
\begin{eqnarray*}
&&\phi_t(x_0,t_0) - c|D\phi(x_0,t_0)| - \textcolor{black}{\mathcal M^+}(D^2\phi(x_0,t_0)) -\\
&&\hspace{1cm}c (\mathcal J^1_{\delta'} [x_0,t_0,\phi]+ \mathcal J^2_{\delta'}[x_0,t_0,D\varphi(x_0,t_0), w])\leq0. 
\end{eqnarray*}
Hence, $w$ is a viscosity subsolution of the equation
$$
w_t - c|Dw| - \textcolor{black}{\mathcal M^+}(D^2w) -c \mathcal J[x,t,w] = 0 \hbox { in } \Omega\times (0,T).
$$
In case the sub and super-solutions are Lipschitz we take $R^* = \max\{||Du||_\infty,||Dv||_\infty\}$ 
and denote by $c = c_{R^*}$ and $w = w_{R^*}$. \smallskip
 
2. The equation satisfies the strong maximum principle since the nonlinearity is positively 1-homogeneous and the nondegeneracy conditions $(N)$ and $(N')$ are satisfied.
\begin{eqnarray*} (N)
 && -c |p| - \textcolor{black}{\mathcal M^+}(I-\gamma p \otimes p) -c\tilde C_\mu + c\gamma\int_{\mathcal{C}_{\eta,\gamma}} | p\cdot j(x,z)|^2 \mu(dz) \geq \\
 && -c |p| - \textcolor{black}{\mathcal M^+(I)+\gamma \mathcal M^-(p \otimes p)} -c\tilde C_\mu + c\gamma\int_{\mathcal{C}_{\eta,\gamma}} | p\cdot j(x,z)|^2 \mu(dz) \geq \\
 && -c |p| - \textcolor{black}{\Lambda N + \lambda \gamma |p|^2} - c\tilde C_\mu + C(\eta) \gamma^{\beta-1}|p|^2 >0, \hbox{ for } \gamma \hbox{ large }.\\
\end{eqnarray*}
Therefore, SMaxP applies and we conclude that if $u- v$ attains a maximum inside the domain $\Omega\times (0,T)$ at some point $(x_0,t_0)$ then $u-v$ is constant in $\Omega\times[0,t_0]$. 
\end{proof} 
\medskip

\begin{remark}
 If Pucci's operator $\textcolor{black}{\mathcal M^+}$ appearing in hypothesis $(H)$ is nondegenerate, 
i.e. $\lambda_R >0$, then one can consider any L\'evy measure $\mu$, not necessarily satisfying $(M_c)$.  
\end{remark}

\begin{example}
 The linear PIDE
 $$
 u_t - a(x)\Delta u - \mathcal I[x,t,u] = f(x) \hbox{ in } \Omega
 $$
 with $a(x)\geq0$, satisfies Strong Comparison, as (H) holds for the corresponding nonlinearity.  
 \end{example}
 
 \begin{example}
  On the other hand, for the equation 
 $$
 u_t + |Du|^m -\mathcal I[u] = f(x) \hbox{ in } \Omega
 $$
with $m\geq 2$ condition $(H)$ holds if the sub and super-solutions are Lipschitz continuous in space.
\smallskip  
 
 \emph{Indeed, for $u$ subsolution and $v$ supersolution
 \begin{eqnarray*}
&&(u-v)_t + |Du|^m - |Dv|^m - \mathcal I[u-v]   \\
&&\hspace{1cm}\geq(u-v)_t + m |Dv|^{m-2} (Du-Dv) - \mathcal I[u-v]\\
&& \hspace{1cm} \geq (u-v)_t -c D(u-v) - \mathcal I[u-v].
 \end{eqnarray*}}
\end{example}

\bigskip\section{Appendix}

We present in the following some useful properties of the nonlocal terms. For a given function $v$ defined on $\R^N\times[0,T]$, consider the integral operators 
$$
\mathcal{I}[x,t,v]=\int_{\mathcal D} (v(x+z,t)-v(x,t) - Dv(x,t)\cdot z 1_B(z))\mu_x(dz),
$$
and 
$$
\mathcal{J}[x,t,v]=\int_{\mathcal D} (v(x+j(x,z),t)-v(x,t) - Dv(x,t)\cdot j(x,z)1_B(z))\mu(dz),
$$
where the integral is taken over a domain $\mathcal D\subseteq \R^N$.

\begin{lemma}\label{lemma:exp_ineq1}
Any smooth function $$v(x,t)=e^{\varphi(x,t)}$$ satisfies the integral inequality
\begin{equation*}
 \mathcal{I}[x,t,v]\geq v\cdot \mathcal{I}[x,t,\varphi], \forall (x,t)\in\R^N\times[0,T]
\end{equation*}
\end{lemma}

\begin{proof}
The inequality is immediate from $e^y - 1\geq y$, $\forall y\in\R$. More precisely
\begin{eqnarray*}
\mathcal{I}[x,t,v]&  =  & \int_{\mathcal D} \big(e^{\varphi(x+z,t)}-e^{\varphi(x,t)} -e^{\varphi(x,t)}D\varphi(x,t)\cdot z 1_B(z)\big)\mu_x(dz)\\
       \ &  =  & e^{\varphi(x,t)}\int_{\mathcal D} \big(e^{\varphi(x+z,t) - \varphi(x,t)}-1 -D\varphi(x,t)\cdot z 1_B(z)\big)\mu_x(dz)\\
       \ & \geq & e^{\varphi(x,t)}\int_{\mathcal D} \big(\varphi(x+z,t) - \varphi(x,t)-D\varphi(x,t)\cdot z 1_B(z)\big) \mu_x(dz).
\end{eqnarray*}
\end{proof}

\noindent We straighten the convex inequality to the following:
\begin{lemma}\label{lemma:exp_ineq2}
Let $v$ be a smooth function of the form  $$v(x,t)=e^{\varphi(x,t)}.$$ Then for any $\delta\geq 0$ there exists a constant $c=\frac12 e^{-\delta}$ such that $v$ satisfies
\begin{equation*}
 \mathcal{I}[x,t,v]\geq e^{\varphi(x,t)}\cdot [\mathcal{I}[x,t,\varphi]+ c\int_{\mathcal D}(\varphi(x+z,t)-\varphi(x,t))^2\mu_x (dz)], 
\end{equation*}
for all $(x,t)\in\R^N\times[0,T]$, where the integral is taken over the domain
$\mathcal D = \{\varphi(x+z)-\varphi(x)\geq -\delta\}.$
\end{lemma}

\begin{proof} The proof is direct application of the exponential inequality
$$ e^y - 1\geq y+ c y^2, \forall y\geq -\delta .$$
We now insert the previous inequality  with $y=\varphi(x+z,t) - \varphi(x,t)$ in the nonlocal term and obtain 
\begin{eqnarray*}
\mathcal{I}[x,t,e^{\varphi}]&  =  & e^{\varphi(x,t)}\int_{\mathcal D} \big(e^{\varphi(x+z,t) - \varphi(x,t)}-1 -D\varphi(x,t)\cdot z  1_B(z)\big)\mu_x(dz)\\
                    \ & \geq & e^{\varphi(x,t)} [\int_{\mathcal D} \big(\varphi(x+z,t) - \varphi(x,t)-D\varphi(x,t)\cdot z  1_B(z)\big)\mu_x(dz)\\
 	              \ & \    & \hspace{0.7cm}  + c\int_{\mathcal D} \big((\varphi(x+z,t) - \varphi(x,t)\big)^2\mu_x(dz)].
\end{eqnarray*}
\end{proof}
    
Similar results hold for L\'evy-It\^o operators.
\begin{lemma}\label{lemma:exp_ineq1_LI}
The function $v(x,t)=e^{\varphi(x,t)}$, satisfies the integral inequality
\begin{equation*}
 \mathcal{J}[x,t,v]\geq v\cdot \mathcal{J}[x,t,\varphi], \forall (x,t)\in\R^N\times[0,T].
\end{equation*}
\end{lemma}

\begin{lemma}\label{lemma:exp_ineq2_LI}
For any $\delta\geq 0$ there exists a constant $c=\frac12 e^{-\delta}$ such that $v=e^{\varphi}$ satisfies
\begin{equation*}
 \mathcal{J}[x,t,v]\geq e^{\varphi(x,t)}\cdot [\mathcal{J}[x,t,\varphi]+ c\int_{\mathcal D}(\varphi(x+j(x,z),t)-\varphi(x,t))^2\mu(dz)], 
\end{equation*}
for all $(x,t)\in\R^N\times[0,T],$
where the integral is taken over 
$\mathcal D = \{\varphi(x+j(x,z))-\varphi(x)\geq -\delta\}. $
\end{lemma}

\paragraph*{\textbf{Acknowledgments}}

The author would like to express her warmest thanks to Professor Guy Barles, whose expert guidance on the topics of this work has proved invaluable and whose careful suggestions helped improving the presentation. Many thanks are addressed to Emmanuel Chasseigne and Cyril Imbert for useful discussions and for their interest in this work. This research is partially financed by the MISS project of Centre National
d'Etudes Spatiales, the Office of Naval research under grant N00014-97-1-0839 and by the European Research Council, advanced grant ``Twelve labours''.

\bibliographystyle{plain}

\end{document}